\documentclass[a4paper,11pt]{article}

\usepackage[utf8]{inputenc}

\usepackage{amssymb}
\usepackage{amsthm}
\usepackage{amsmath}
\usepackage{anysize}
\usepackage{url}
\usepackage{color}
\usepackage{graphicx}
\usepackage{overpic}
\usepackage{subcaption}
\usepackage{paralist}
\usepackage{comment}
\usepackage[all]{nowidow}

\newcommand{\R}{\mathbb{R}}

\newtheorem{theorem}{Theorem}
\newtheorem{remark}[theorem]{Remark}

\newtheorem{definition}[theorem]{Definition}

\newtheorem{lemma}[theorem]{Lemma}

\newtheorem{example}[theorem]{Example}
\newtheorem{innercustomthm}{Theorem}
\newenvironment{customthm}[1]
  {\renewcommand\theinnercustomthm{#1}\innercustomthm}
  {\endinnercustomthm}
\numberwithin{theorem}{section}
\numberwithin{equation}{section}

\DeclareMathOperator{\suppp}{supp \,}

\DeclareMathOperator{\var}{var}
\DeclareMathOperator{\real}{real}
\DeclareMathOperator{\dist}{dist\,}

\newcommand{\norm}[2]{\left\|#1\right\|_{#2}}

\renewcommand{\i}{\mathrm{i}}

\title{Stable Phase Retrieval in Infinite Dimensions}
\author{Rima Alaifari, Ingrid Daubechies, Philipp Grohs, Rujie Yin}
\date{}

\begin{document}

\maketitle

\begin{abstract}
The problem of phase retrieval is to determine a signal $f\in \mathcal{H}$, with $
\mathcal{H}$ a Hilbert space, from intensity measurements
$|F(\omega)|$, where $F(\omega):=\langle f , \varphi_\omega\rangle$ are measurements of $f$ with respect to a measurement system $(\varphi_\omega)_{\omega\in \Omega}\subset \mathcal{H}$. 

Although phase retrieval is always stable in the finite dimensional setting whenever it is possible (i.e. injectivity implies stability for the inverse problem), the situation is drastically different if $\mathcal{H}$ is infinite-dimensional: in that case phase retrieval is never uniformly stable \cite{cahill2016phase,grohsstab}; moreover the stability deteriorates severely in the dimension of the problem \cite{cahill2016phase}. 

On the other hand, all empirically observed instabilities are of a certain type: they occur whenever the function $|F|$ of intensity measurements is concentrated on disjoint sets $D_j\subset \Omega$, i.e., when $F= \sum_{j=1}^k F_j$ where each $F_j$ is concentrated on $D_j$ (and $k \geq 2$). 

Motivated by these considerations we  propose a new paradigm for stable phase retrieval by considering the problem of reconstructing $F$ up to a phase factor that is not global, but that can be different for each of the subsets $D_j$, i.e., recovering $F$ up to the equivalence
$$
	F \sim \sum_{j=1}^k e^{\i \alpha_j} F_j.
$$

We present concrete applications (for example in audio processing) where this new notion of stability is natural and meaningful and
show that in this setting stable phase retrieval can actually be achieved, for instance if the measurement system is a Gabor frame or a frame of Cauchy wavelets. 
\end{abstract}
\section{Introduction}
\subsection{Problem Formulation}
Suppose we are given a complex-valued function $F:\Omega\to \mathbb{C}$
on some (discrete or continuous) domain $\Omega$,
and we can observe only its absolute values $|F|$. 
The problem of phase retrieval is to reconstruct $F$ from 
these measurements, up to a global phase (meaning that the functions 
$F$ and $e^{\i \alpha}F$, $\alpha \in \mathbb{R}$, are not distinguished). 

Such problems are encountered in a wide variety of applications,
ranging from X-ray crystallography and microscopy to audio processing and 
deep learning algorithms \cite{shechtman2015phase,deller1993discrete,hurt2001phase,waldspurger2015these}; accordingly, 
a large body of literature treating the mathematical and algorithmic solution of phase retrieval problems exists, with new approaches emerging in recent years \cite{waldspurger2015phase,candes2015phase,balan2006signal}.

In these applications,
 the domain of definition $\Omega$ is a finite
 set, for example $\Omega = \{1,\dots , N\}$ and the function $F$
 arises from a finite number of linear measurements
 $$
 	F(k)=\langle x, a_k\rangle := \sum_{l=1}^d x_l\,\overline{(a_k)_l},\quad 
	a_k \in \mathbb{C}^d,\ k=1,\dots , N
 $$
 of some signal $x\in \mathbb{C}^d$ which one seeks to recover.  
 Such problems arise as finite approximations to various real-world problems; in diffraction imaging, for instance,
 the setup can be interpreted as measuring the diffraction 
 pattern of $x$, modulated with a number of different filters.
 
 Classically, the numerical solution of phase retrieval problems
 is treated via alternating projection algorithms that are simple to 
 implement but lack a theoretical understanding \cite{fienup1982phase,gerchberg1972practical}.
 More recent work \cite{candes2015phase} has introduced an algorithm named PhaseLift, based on a reformulation of the N-dimensional phase retrieval problem as a semidefinite optimization problem in an $N^2-$dimensional space.  
As shown in \cite{candes2015phase}, PhaseLift succeeds with high probability in recovering the signal $x$, up to a global phase, in a randomized setting (meaning that the vectors $a_1,\dots , a_N$ are drawn at random);
 moreover, PhaseLift is stable if the measurements $|\langle x, a_n\rangle|$ are corrupted by additive noise. 
More recently it has been shown that gradient descent algorithms, together with a careful guess for their starting value, achieve the same theoretical guarantees while being vastly more efficient \cite{candes2015phasewirt}.
 
 \subsection{Infinite-Dimensional Phase Retrieval}
 The vector $x\in \mathbb{C}^d$ typically arises as a digital representation of a physical quantity.
 For instance, $x$ could represent a finite-dimensional approximation of a continuous function describing an infinite-dimensional object.
 This naturally leads one to consider the more general infinite-dimensional phase retrieval problem, where one seeks to recover a 
 signal $f\in \mathcal{H}$, with $\mathcal{H}$ a (possibly infinite-dimensional) Hilbert space,  
 from the phaseless measurements  $|F(\omega)|$, with
 \begin{equation}\label{eq:phaseret}
 	F(\omega):=\langle f, \varphi_\omega\rangle,\quad \omega\in \Omega,
 \end{equation}
 and where $(\varphi_\omega)_{\omega\in \Omega} \subset \mathcal{H}$ is a (possibly infinite) parameterized family of measurement functions, typically normalized so that $\|\varphi_\omega\|=1$ for all $\omega \in \Omega$.
 
 We mention a few examples. 
 \begin{itemize}
 \item
Consider the classical $n$-dimensional phase retrieval problem of reconstructing a function $f$ from intensity measurements of its Fourier transform $\widehat{f}$.
For a compact subset $D\subset \mathbb{R}^n$
 let $\mathcal{H}=L^2(D)$ and consider $f \in \mathcal{H}$. Let $F(\omega)= \widehat f(\omega)$, $\omega \in \Omega$, 
where 
 $\Omega$ is either all of $\mathbb{R}^n$ or a suitable discrete subset of $\mathbb{R}^n$ (since $f$ has compact support, there exist $\varphi_\omega \in \mathcal{H}$ such that $F(\omega)=\langle f,\varphi_\omega\rangle$).   
Applications of this setup include coherent diffraction imaging, X-ray crystallography and many more, in which one typically can measure only intensities, corresponding to $|\langle f,\varphi_\omega\rangle|^2$. 
The classical phase retrieval problem is in general not uniquely solvable \cite{akutowicz1957determination}; recent work
\cite{pohl2014phaseless} has established the uniqueness of the solution, if the intensities of the Fourier transforms of certain structured modulations of $f$ are measured instead.

\item Related to the previous example, the work \cite{thakur2011reconstruction} studies the reconstruction 
of a bandlimited \emph{real-valued} function $f$ from unsigned
samples $(|f(\omega)|)_{\omega\in \Omega}$ with $\Omega$ a suitable (discrete) sampling set; more general settings are
considered in \cite{alaifari2016reconstructing,chen2016phase}. 
Note that the real-valued case (where only the sign $\pm 1$ is missing
from each measurement) is qualitatively simpler than the complex-valued
case where each measurement lacks a phase factor $e^{\i\alpha}$, $\alpha \in \mathbb{R}$.

\item In order to overcome the problem of nonuniqueness of the classical phase retrieval problem and to be able to apply techniques in diffraction imaging also to extended objects, one often records local illuminations of different overlapping parts
of the object, which mathematically amounts to a windowed (or short-time) Fourier transform (STFT)
 $
 F = V_g f,
 $
 where for $f\in L^2(\mathbb{R})$
 \begin{equation}\label{eq:wft}
 V_g f(x,y) :=\int_{\mathbb{R}}f(t)\overline{g(t-x)}e^{-2\pi \i t y}dt
 \end{equation} is defined by the window $g\in L^2(\mathbb{R})$ and the parameters $(x,y)$ may vary over a discrete or continuous
 subset of $\mathbb{R}^2$. See \cite{shechtman2015phase} for an excellent survey on phase retrieval from STFT measurements.

\item Another instance of phase retrieval from STFT measurements arises in audio processing applications involving phase vocoders. A phase vocoder \cite{flanagan1966phase} is a tool that allows to modify an audio signal $f$ by transforming its STFT.
Given $f$, a phase vocoder first calculates $V_g f(x,y)$ and then modifies it to some $H(x,y)$ before it transforms back to the time domain by taking the inverse (discrete) STFT of $H$. Typical modifications include time-scaling and pitch shifting. In general, the modified $H$ may not result in an STFT of \emph{any} signal. This leads to the so-called \emph{phase coherence} problem \cite{laroche1999improved} in which one aims to make modifications such that the modified $H$ is an approximate STFT. One possible approach is to modify the amplitude $|F(x,y)|$ only in a first step to obtain $|H(x,y)|$ and to then recover the phase of $H(x,y)$ in a coherent way.

  \item More recent work \cite{waldspurger2015these}
  seeks to reconstruct  a signal $f\in L^2(\mathbb{R})$ from the magnitudes $|F(x,2^j)|$ of semidiscrete
  wavelet measurements, where 
  $
  F(x,2^j) = |w_\psi f(x, 2^j)|,
  $
  with $j\in \mathbb{N}$, $x\in \mathbb{R}$ 
  and $w_\psi f(x,y):=\int_{\mathbb{R}}f(t){|y|^{1/2}}\overline{\psi(y(t-x))}dt$ \footnote{Note that our $w_\psi f(x,y)$ corresponds to $W_\psi f(x,1/y)$ in the notation of \cite{waldspurger2015these}.}; the collection of these magnitudes is sometimes called the \emph{scalogram}.  
  The corresponding phase retrieval problem arises in e.g. the reconstruction of $f$
  from the output of its so-called scattering transform as defined
  in \cite{mallat2012group}.
  \end{itemize}
  
  In all these examples it is extremely challenging to establish  
  whether $f$ is uniquely determined, up to a global phase; 
  the problem is still not well understood except in special cases.
  
  \subsection{(In-)stability of (In-)finite Dimensional Phase Retrieval}\label{sec:instab}
  Even if the uniqueness problem was completely solved, this would however not yet be sufficient for applications.
  Since physical measurements are always corrupted by noise and/or
  uncertainties and numerical algorithms always introduce rounding errors,
  solving a real-world phase retrieval problem mandates a
  reconstruction that is stable, meaning that there should
  exist a (moderate) constant $C>0$ such that
  \begin{equation}\label{eq:oldspab}
  	\inf_{\alpha\in \mathbb{R}}\norm{F - e^{\i\alpha}G}{\mathcal{B}}\le C\norm{|F|-|G|}{ \mathcal{B}'},
  \end{equation}
  for $\mathcal{B}$, $\mathcal{B}'$ suitable Hilbert (or Banach) spaces.
  
  For phase retrieval problems in spaces of finite (and fixed) dimensions, stability and uniqueness typically go hand in hand \cite{bandeira2014saving,candes2015phase}. The situation changes drastically when we consider infinite-dimensional spaces. A central finding of \cite{grohsstab, cahill2016phase} is that
  \emph{all infinite-dimensional phase retrieval problems are 
  unstable} and that \emph{the stability of finite-dimensional phase retrieval problems deteriorates severely as the dimension grows}.
  
  \begin{example}[Stability deterioration as the dimension grows]\label{ex:cahill}
  We borrow the following example from recent work \emph{\cite{cahill2016phase}} to which we refer for more detail.
  Consider the real-valued Paley-Wiener space
$$
	PW = \{ f \in L^2(\mathbb{R},\mathbb{R}): \suppp \widehat{f} \subseteq [-\pi,\pi]\},
$$
and the 
measurement vectors $\{ \varphi_n\}_{n \in \mathbb{Z}}$ of elements
$\varphi_n := \mbox{sinc}( \cdot - \frac{n}{4})$. 
As shown in \emph{\cite{thakur2011reconstruction}}, each $f\in PW$ is uniquely determined by $\{ |\langle f, \varphi_n\rangle| \}_{n \in \mathbb{Z}}$, 
 up to a global sign $\pm 1$ (note that this setup is real-valued). More precisely, suppose
that $f,g\in PW$ with $|\langle f , \varphi_n\rangle | = |\langle g , \varphi_n\rangle |$ for all $n\in \mathbb{Z}$. 
Then there exists $\sigma \in \{-1,1\}$ with $f=\sigma \cdot g$.

Now we consider an approximate problem restricted
to the finite-dimensional subspaces $V_n \subset PW$, defined as 
$$
	V_n := \text{span } \{ \varphi_{4 \ell}: \ell \in [-n,n]\}.
$$
The space $V_n$ consists of $f\in PW$ for which $\widehat{f}$ is the restriction to $[-\pi,\pi]$ of a trigonometric polynomial of degree $n$.
Then, \emph{\cite{cahill2016phase}} gives the explicit construction of $f_m,\,g_m\in V_{2m}$ such that, for some $m-$independent constant $c>0$,
\begin{equation}\label{eq:cahillex}
	\min_{\tau \in \{ \pm 1\}} \norm{f_m - \tau g_m }{L^2(\mathbb{R})} > c (m+1)^{-1} 2^{3m} \norm{ (|\langle f, \varphi_n\rangle|-|\langle g, \varphi_n\rangle|)_{n \in \mathbb{Z}}}{\ell^2(\mathbb{Z}) },\quad {\forall m \in\mathbb{N}.}
\end{equation}
Comparing this with \eqref{eq:oldspab}, we find that the corresponding Lipschitz constant $C$ thus decays at least exponentially fast as the dimension of the problem grows.
  \end{example}
\begin{figure}[h!]
	\centering
	\begin{minipage}[c]{.45\textwidth}
	\includegraphics[width = \textwidth]{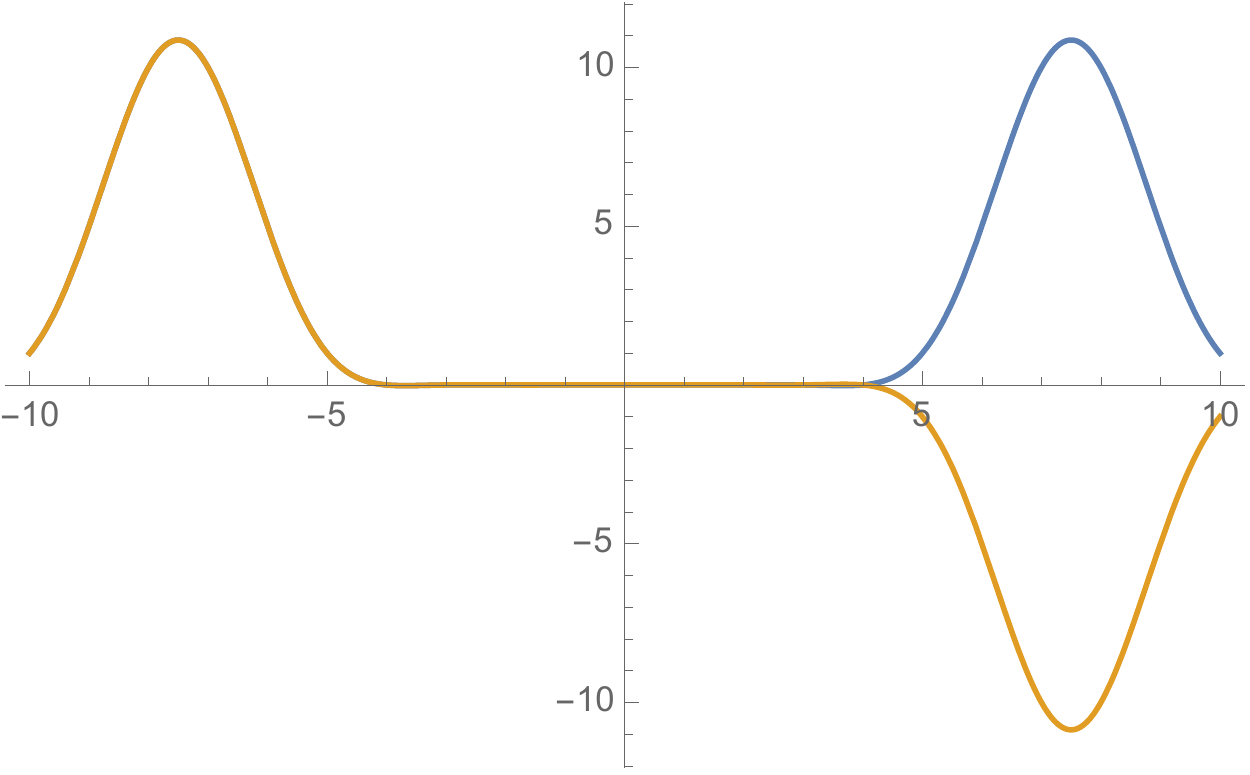}
	\end{minipage}
	\hfill
	\begin{minipage}[c]{.5\textwidth}
	\caption{Functions $f_5$ and $g_5$ satisfying (\ref{eq:cahillex}) and with 
	$\suppp\widehat{f}_5 = \suppp \widehat{g}_5 = [-\pi,\pi]$. While $\sup_{n\in\mathbb{Z}}\left||f_5(\frac{n}{4})| -|g_5(\frac{n}{4})|\right|$ is small, $\| f_5- g_5\|_{L^2(\mathbb{R})}$ and $\|f_5+g_5\|_{L^2(\mathbb{R})}$ are not.}
	\label{fig:cahill}
	\end{minipage}
\end{figure}

Figure \ref{fig:cahill} shows the plot of the functions $f_n$ and
$g_n$ for $n=5$, illustrating that the two functions
have almost identical absolute value despite being significantly different from each other. Note that the two functions in this example are large on two distant domains and small in between. It turns out, as shown in \cite{grohsstab}, that in the real-valued setting this is the 
generic form of instabilities: similar pairs of functions can
be constructed in much more general settings, in particular for all previously mentioned examples of phase retrieval problems. Consequently,
\emph{stable phase retrieval is not possible for infinite-dimensional problems, 
or even for their fine-grained (and thus finite- but high-dimensional) approximations}.

\subsection{Three Observations and A New Paradigm}\label{sec:new}
The instability for infinite-dimensional phase retrieval problems and for their high-resolution approximations
makes one wonder whether phase retrieval is even advisable in these situations.
It is instructive, however, to take a closer look at how this instability manifests itself in concrete phase retrieval attempts. We offer the following three observations.

\begin{enumerate}
\item 
One way to construct phase retrieval problems leading to instabilities is to consider 
functions $F= \sum_{j=1}^k F_j$
with $F_j$ concentrated on disjoint sets $D_j$ that are far apart from each other. In the sequel, we will occasionally refer to functions of this form as multi-component functions.
Clearly, any function of the form
\begin{equation}\label{eq:bulkphase}
	G:=\sum_{j=1}^k e^{\i \alpha_j}F_j
\end{equation}
for any $\alpha_1,\dots , \alpha_k\in \mathbb{R}$,
will result in an instability: the absolute values of $F,G$ will be very close, 
due to the fact that the $F_j$'s are concentrated on well-separated disjoint sets, but $F-e^{\i \gamma} G$ need not be small at all, even for the optimal choice of $\gamma$.

The functions constructed in Example \ref{ex:cahill} are of this form with $k=2$. In fact, in the general real-valued case it can be shown that \emph{all} instabilities arise
in this way \cite{grohsstab}. In the complex case, it is not known whether this is the case as well.

\item 
One can investigate how existing concrete phase retrieval algorithms deal with finite-dimensional approximations to the multi-component $F$ introduced above, under item 1.
Figure \ref{fig:walds} gives a typical albeit simplistic example. Consider an analytic\footnote{i.e. $\widehat{f}(\omega)  = 0, \forall \omega < 0$,} signal $f$, e.g. as in Figure \ref{fig: walds-1},  whose Gabor transform $F = V_\varphi f$ (as in Definition \ref{eq:wft} with $\varphi = e^{-\pi t^2}$) has two disconnected components $F_1,\, F_2$, s.t. $F= F_1 + F_2$, see Figure \ref{fig: walds-2}. 
Given the Gabor transform measurements $|F| = |V_{\varphi} f|$, a reconstruction $f^{rec}$ is obtained using the phase retrieval algorithm in \cite{waldspurger2015these}, and the corresponding code from \protect\url{http://www.di.ens.fr/~waldspurger/wavelets_phase_retrieval.html}\footnote{The original algorithm works on magnitude measurements of wavelet transforms such as Morlet wavelets and Cauchy wavelets. Here we apply it to dyadic Gabor wavelet, where the phenomenon of phase difference between the initial and reconstructed signal persists.}. The relative error $\Vert f-f^{rec}\Vert/\Vert f\Vert$ in time domain is 8.61$\times 10^{-1}$ whereas the relative error $\Vert |F|-|F^{rec}|\Vert / \| F\|$ in the Gabor transform measurements is 1.27$\times 10^{-5}$. The large difference in the time domain (the ratio of the relative errors exceeds $5\times 10^{4}$; see also Figure \ref{fig: walds-3}), is due to a non-uniform but piecewise constant phase shift in the time-frequency domain. 
Let $F_1^{rec},\, F_2^{rec}$ be the two components of $F^{rec}$ corresponding to $F_1,\,F_2$. As shown in Figure \ref{fig: walds-4}, $F_1$ and $F_1^{rec}$ differ by only a phase factor $e^{\i\alpha_1}$; similarly $F_2$ and $F_2^{rec}$ differ by $e^{\i\alpha_2}$; however, $\alpha_1\neq \alpha_2$. 
So although it is hopeless to expect that any numerical algorithm could stably distinguish such a multi-component function from $\sum_{j=1}^k e^{\i \alpha_j}F_j$, algorithmic reconstruction \emph{up to the equivalence $\sum_{j=1}^k F_j\sim \sum_{j=1}^k e^{\i \alpha_j}F_j$} seems to work quite well.

\item 
Being able to reconstruct (if this is indeed feasible) multi-component functions of the type $\sum_{j=1}^k F_j$ up to the equivalence $\sum_{j=1}^k F_j \sim \sum_{j=1}^k e^{\i \alpha_j} F_j$ is of interest only if this equivalence is itself meaningful.

Our third observation is that this is indeed the case for some applications. We list two examples here.\\
Our first example is concerned with coherent diffraction imaging. Measurements
of X-ray diffraction intensities  by complicated objects allow reconstruction of the object under 
certain constraints on the object; see, e.g. \cite{klibanov1986inverse} 
for a mathematical uniqueness
result, or \cite{marchesini2003x} 
for an algorithm effective for fine-grained reconstruction on physical 
data sets that are supported in a finite volume, without the exact location
of this support being known. 
In its most stripped-down form, the problem consists in reconstruction of a function
$f$ supported on a compact domain $\Omega$ from measurements of the magnitude of its Fourier transform,
$|\widehat{f}(\xi)|$. For the plain-vanilla scattering implementation, the physical object 
to be reconstructed is illuminated by a plane wave. If the object is more extended,
illumination by more narrowly concentrated beams might be easier to achieve; one then acquires
scattering intensity data for each of several different beam illuminations, which corresponds 
to replacing the Fourier transform by an STFT. The methodology which we just described is widely used; for example
in Fourier Ptychography \cite{humphry2012ptychographic,rodenburg2008ptychography,zheng2013wide}.

If the scene to be reconstructed
consisted of several disjoint objects, separated by ``empty'' space (the example in Figure 1 
in \cite{marchesini2003x} illustrates such an example), then reconstruction of the individual objects might
be numerically and mathematically much easier if it were allowed to reconstruct each object 
up to a uniform
phase (for complex $f$) or up to a uniform sign (for real $f$).  The simulation illustrated in 
Figure \ref{fig:walds}, for a 1-dimensional Gabor transform, suggests as much.

Our second example is concerned with audio processing. It is well known that human audio perception is insensitive to a ``global phase change''. One way to show this is
to start with a (real-valued) audio signal $f(t)$, with Fourier transform $\widehat{f}(\xi)$,
and carry out the following operations: first, take its analytic representation  
$f_a$ by disrecarding its negative frequency components: $\widehat{f_a}(\xi):=\widehat{f}(\xi)\chi_{\xi>0}$;
next multiply it by an arbitrary (but fixed) phase $e^{\i\alpha}$, $\widehat{f_a^\alpha} (\xi):= e^{\i\alpha} \widehat{f_a}(\xi)$. Finally we turn it back into the Fourier
transform of a real-valued function $f^{\alpha}$ by ``symmetrizing'', i.e. by setting 
$\,\widehat{f^{\alpha}}(\xi)\,= \,e^{\i\alpha}\widehat{f}(\xi)\chi_{\xi>0}\,+\,
e^{-\i\alpha}\overline{\widehat{f}(-\xi)}\chi_{\xi<0}\,$ (note that $\overline{\widehat{f}(-\xi)} = \widehat{f}(\xi)$ because $f$ is real-valued). Equivalently, $f^{\alpha}$ can be
expressed in terms of the original signal $f$ as  
$f^{\alpha}(t)=\,\cos\alpha\cdot f(t) \,+\, \sin\alpha\cdot (Hf)(-t)$,
where $Hf$ is the Hilbert transform of $f$. Then, even though the plot of $f$ is typically very different 
from that of $f^\alpha$ (if $\alpha$ differs significantly from a multiple of $2\pi$),
the two sound the same to the human ear, making them equivalent for most practical applications. 
Consider now an audio signal $f$ consisting of two ``bursts'' of sound, separated by a short stretch of silence, i.e. $f(t)\,=\,f_1(t)\,+\,f_2(t)\,$, with $\suppp f_1=[t_1,T_1]$ and 
$\suppp f_2=[t_2,T_2]$ where $t_2-T_1 > \tau$ for some pre-assigned positive $\tau$ (typically of
the order of a few tenths of seconds). Figure \ref{fig: audio2} plots such an example, for the utterance ``cup, luck'',
retrieved from the database at \url{http://www.antimoon.com/how/pronunc-soundsipa.htm}, with ``cup'' corresponding to $f_1$, ``luck'' to $f_2$. Because both
$f_1$ and $f_2$ are highly oscillatory (as is customary for audio signals), 
$Hf_1$ and $Hf_2$ both have fast 
decay, and are negligibly small outside $\suppp f_1=[t_1,T_1]$ and 
$\suppp f_2=[t_2,T_2]$, respectively. For such signals $f$, one can pick two {\em different} phases
$\alpha_1$ and $\alpha_2$, and construct 
$f^{\alpha_1,\alpha_2}=f_1^{\alpha_1}+f_2^{\alpha_2}$; the resulting audio signals again sound exactly the same as the original
$f$. On \url{https://services.math.duke.edu/~rachel/research/PhaseRetrieval/acoustic_result/acoustic_result.html}, one can download and/or listen to $f$ and $f^{\alpha_1,\alpha_2}$.

We further note that signals remain undistinguishable to the human ear under a more general class of transformations: even for signals $f=\sum_{j=1}^J f_j$ with $J>2$ components, in which the $f_j$ correspond to components $F_j$ that are separated in the time-frequency domain (but not necessarily in time, or in frequency) replacing each $F_j$ by $e^{\i \alpha_j} F_j$ results in a signal that sounds exactly like the original signal $f$ (see Figure \ref{fig:Gabormodul} for an example of such a signal and its Gabor transform; on   \url{https://services.math.duke.edu/~rachel/research/PhaseRetrieval/acoustic_result/acoustic_result.html}
one can listen to this example and component-wise phase-shifted versions). 

If one seeks to reconstruct $f$ only within the equivalence class of audio signals that are indistinguishable
from $f$ by human perception, then it is thus natural to treat all the functions of type (\ref{eq:bulkphase}) as equivalent, for all choices of $\alpha_j$.

\end{enumerate}

\begin{figure}[h!]
	\centering
	\begin{subfigure}[t]{.45\textwidth}
	\includegraphics[width = \textwidth]{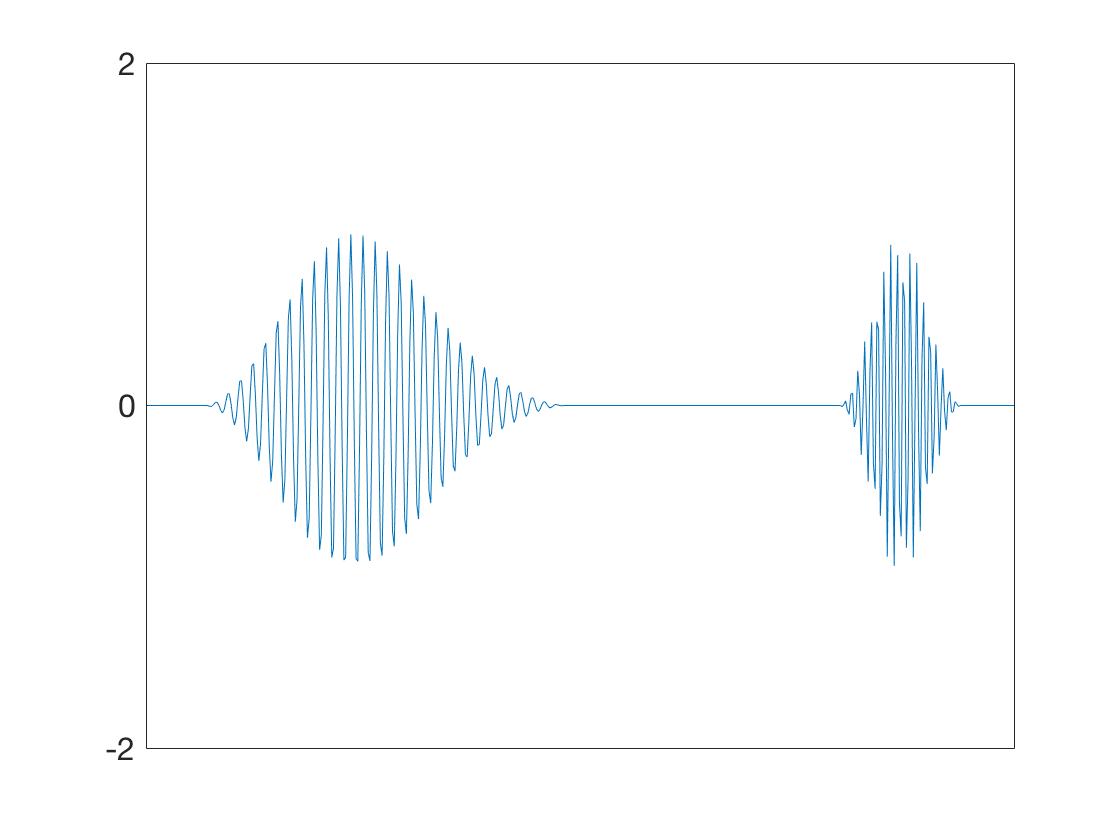}
	\caption{$\real(f)$ {\small in time domain}} \label{fig: walds-1}
	\end{subfigure}
	\hfill
	\begin{subfigure}[t]{.47\textwidth}
	\includegraphics[width = \textwidth]{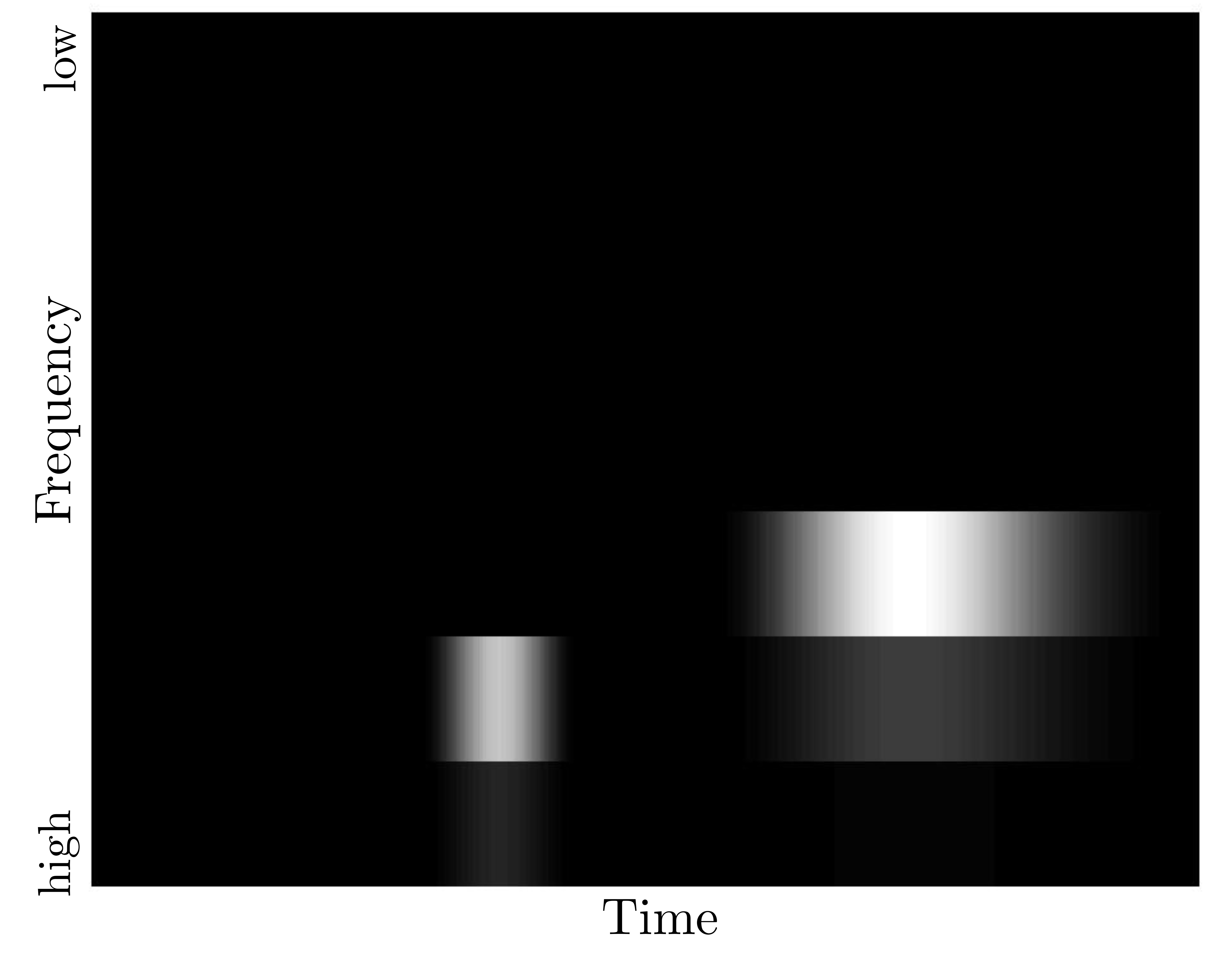}
	\caption{$|V_{\varphi} f|$ {\small in time-frequency (TF) domain of Gabor transform}}\label{fig: walds-2}
	\end{subfigure}\\[1em]
	\begin{subfigure}[t]{.45\textwidth}
	\includegraphics[width = \textwidth]{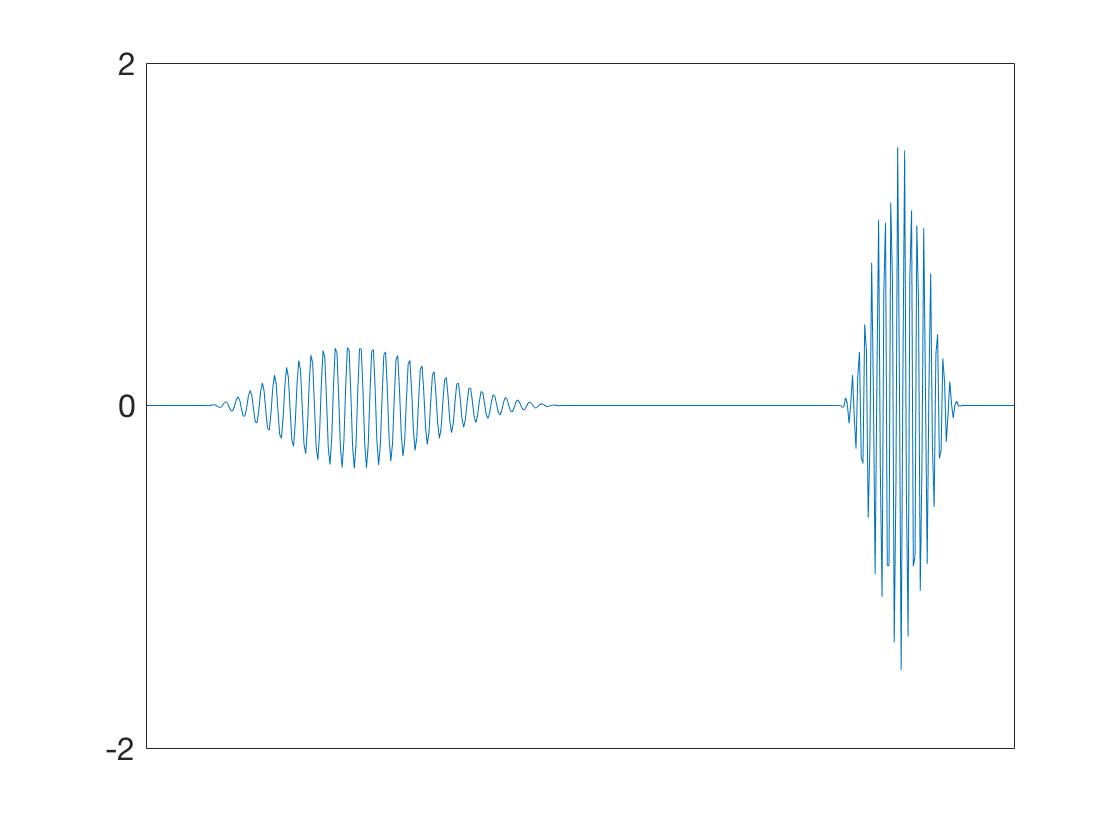}
	\caption{$\real(f-f^{rec})$ {\small in time domain}}\label{fig: walds-3}
	\end{subfigure}
	\hfill
	\begin{subfigure}[t]{.45\textwidth}
	\includegraphics[width = \textwidth]{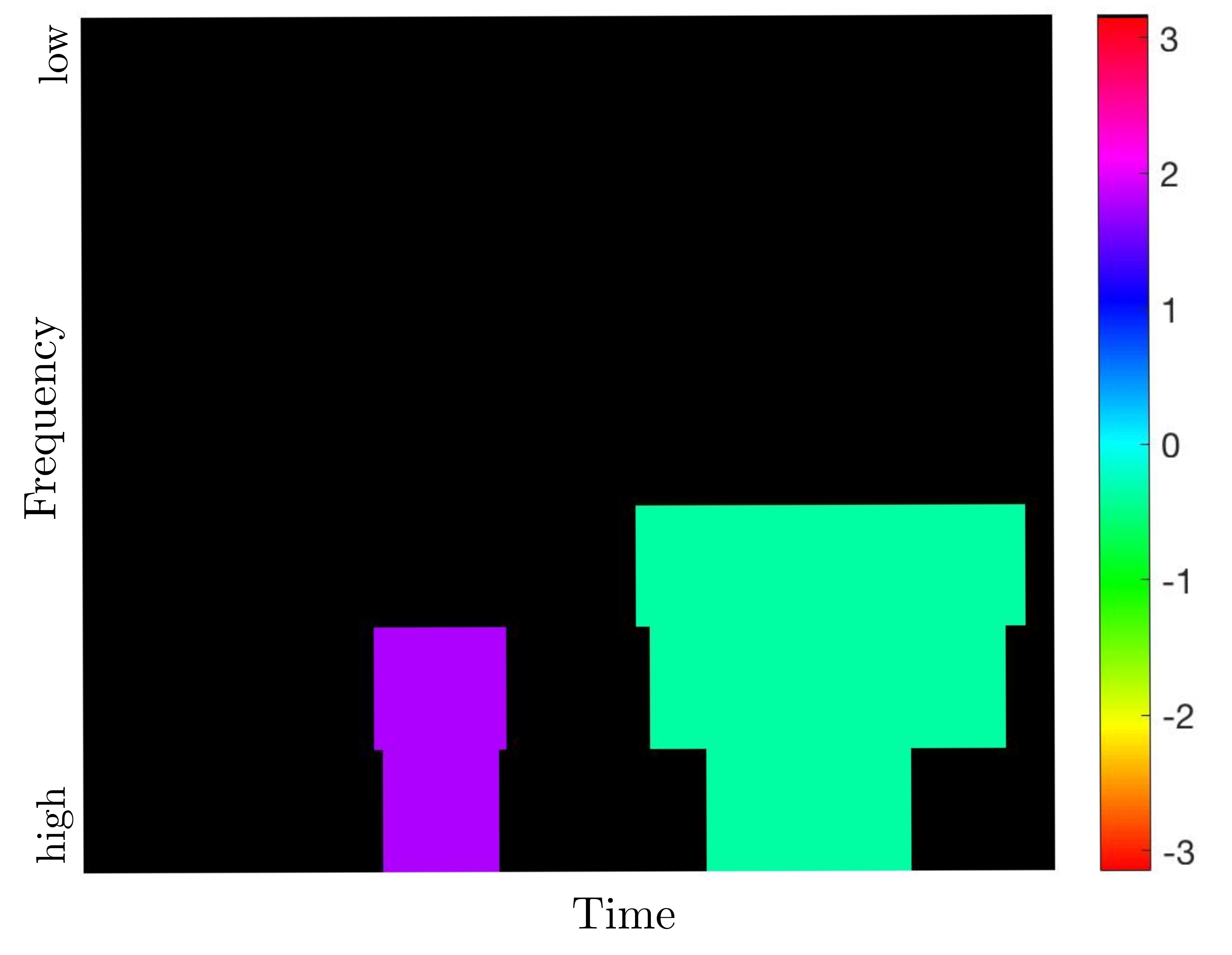}
	\caption{$\arg(V_{\varphi}f/V_{\varphi}f^{rec})$ {\small in TF domain}}\label{fig: walds-4}
	\end{subfigure}
	\caption{ Phase retrieval on the Gabor measurements $|V_{\varphi} f|$ of an analytic function $f$; note that in (\ref{fig: walds-4}), $\arg(V_{\varphi}f/V_{\varphi}f^{rec})=\alpha_j$, on the domain where $F_j$ is large, $j= 1, 2.$ The Gabor measurements  $|V_{\varphi} f|$ consist of two components that are localized and well-separated in time, as illustrated by (\ref{fig: walds-1}) and (\ref{fig: walds-2}). On the measurements of  $|V_{\varphi} f|$ shown in (\ref{fig: walds-2}) we applied the algorithm in \cite{waldspurger2015these} to reconstruct a candidate $f^{rec}$, which is markedly different from $f$, as shown by (\ref{fig: walds-3}). However, a careful analysis of each of the components separately shows that the only difference lies in a different phase factor (see (\ref{fig: walds-4})): $f^{rec} = e^{\i \gamma_1} f_1 + e^{\i \gamma_2} f_2$ for some $\gamma_1  \neq \gamma_2$, whereas $f=f_1+f_2$.}
	\label{fig:walds}
\end{figure}

\begin{figure}
	\centering
	\begin{subfigure}[t]{.45\textwidth}
	\includegraphics[width = \textwidth]{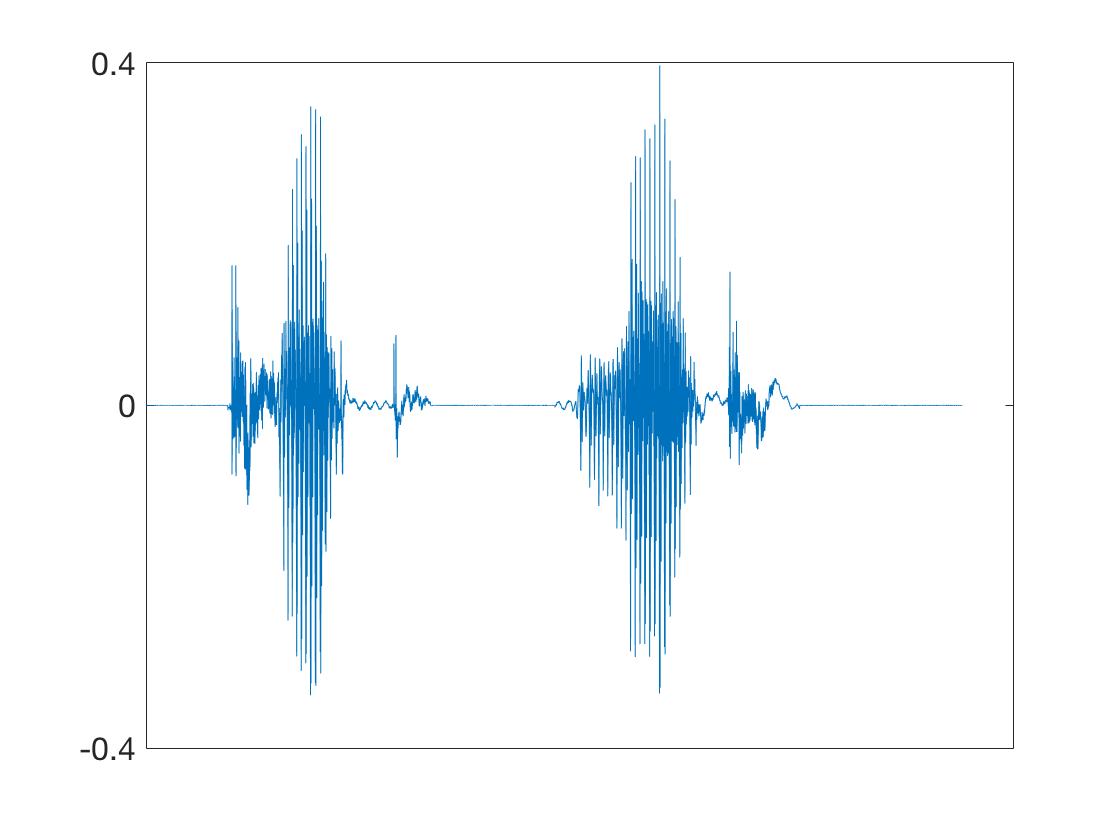}
	\caption{Audio signal $f$ in time domain.}\label{fig: audio2}
	\end{subfigure}
	\hfill
	\begin{subfigure}[t]{.45\textwidth}
	\includegraphics[width = \textwidth]{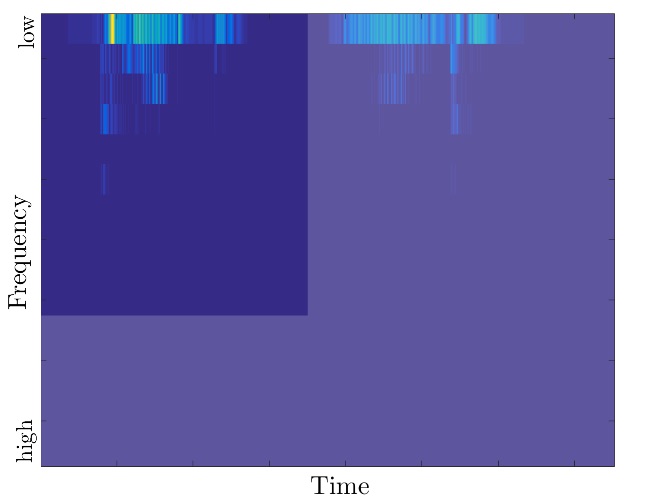}
	\caption{Time-frequency plot of the magnitude $|F|$ of the Gabor representation of $f$, with one separated component highlighted.}\label{fig: audio2-TF}
	\end{subfigure}
	\caption{Audio signal ``cup luck" and its Gabor measurements; both in the time domain and in the time-frequency plane the two components are well-separated.}
\end{figure}

\begin{figure}
	\centering
	\begin{subfigure}[t]{.45\textwidth}
	\includegraphics[width = \textwidth]{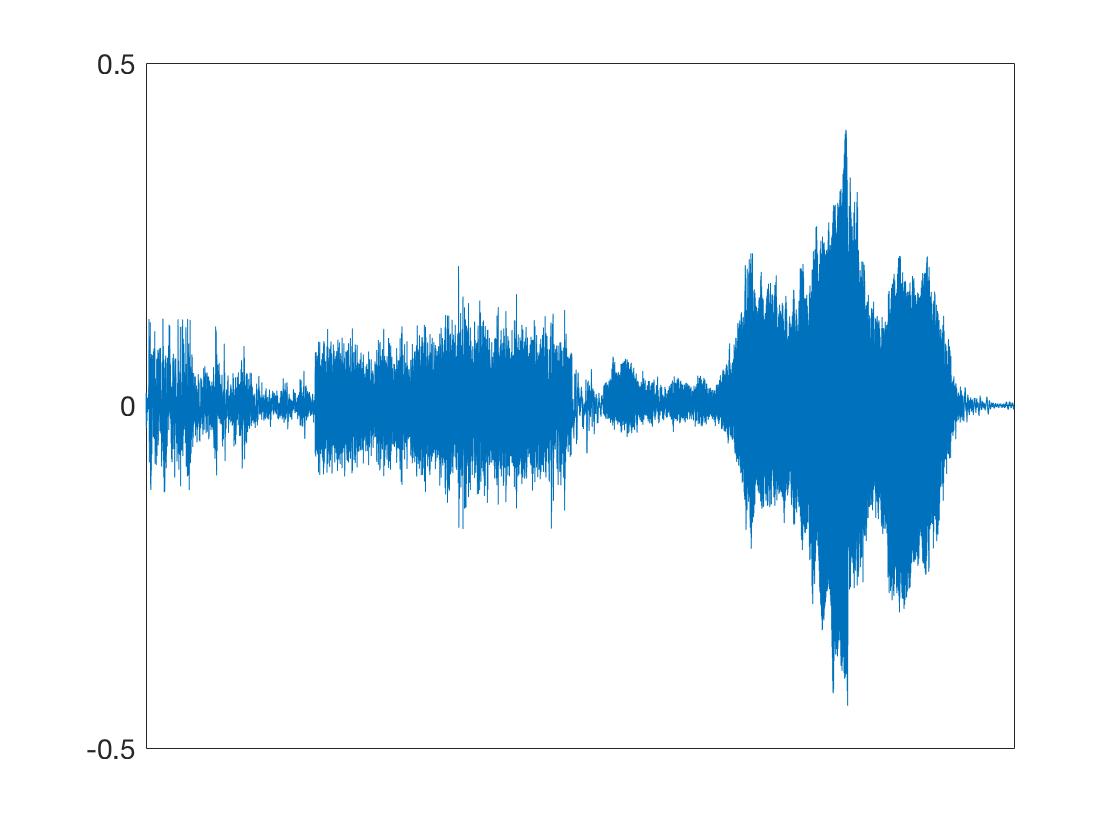}
	\caption{Audio signal $f$ in time domain.}\label{fig: audio1}
	\end{subfigure}
	\hfill
	\begin{subfigure}[t]{.45\textwidth}
	\includegraphics[width = \textwidth]{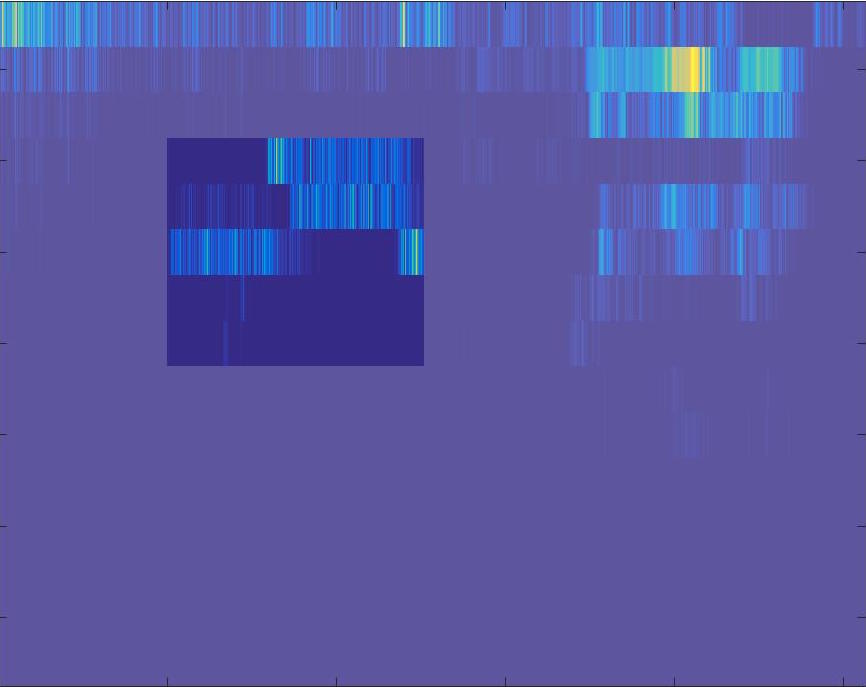}
	\caption{Time-frequency plot of the magnitude $|F|$ of the Gabor representation of $f$, with one separated component highlighted. }\label{fig: audio1-TF}
	\end{subfigure}
	\caption{Audio signal of a sound mixture of thunder, a bird call and a baby crying, together with its Gabor measurements; although in this example
		there is no clear separation in either time or frequency, one can carve out separated components in the time-frequency plane (one of them is highlighted here).}
		\label{fig:Gabormodul}
\end{figure}

These observations suggest a new paradigm for stable phase retrieval: rather than aiming for
bounds of the form (\ref{eq:oldspab}) (which we know do not exist), we investigate a weaker form of stability
that would be sufficient for this type of application: we study the stability of phase retrieval \emph{subject to the equivalence $\sum_{j=1}^kF_j \sim \sum_{j=1}^k
e^{\i \alpha_j}F_j$}, 
that is, bounds of the form 
  \begin{equation}\label{eq:newform}
  	\inf_{\alpha_1,\dots , \alpha_k \in \mathbb{R}}
  	\sum_{j=1}^k\norm{F_j - e^{\i\alpha_j}G_j}{\mathcal{B}} \le C\norm{|F|-|G|}{ \mathcal{B}'},
  \end{equation}
  where $\mathcal{B}$, $\mathcal{B}'$ are suitable Hilbert (or Banach) spaces and $F_j, G_j$ any pairs of functions which have their essential support contained in sets $D_j$. 
  
The question of whether bounds of the form (\ref{eq:newform}) can actually be established for examples of practical interest will be the main subject of this article.
  
\subsection{Stability for Atoll Functions}\label{sec:atoll}

To study this question mathematically, we first need to make it more precise. 
Figure \ref{fig:Gabormodul} suggests that a realistic model for Gabor transform measurements on acoustic signals
are functions $\sum_{j=1}^kF_j$ where each $F_j$ is ``large'' on a domain $D_j$, which we shall interpret as a strictly positive lower bound on $|F_j|$. In practice, we expect that $F_j$ may still have zeroes within $D_j$, which means that there could be ''holes'' in $D_j$ (reasonably small neighborhoods of these zeroes) on which $|F_j|$ could \emph{not} be bounded below away from zero. This motivates the following definition:

\begin{definition}[Atoll domains]\label{def:holes}Let $D\subset \mathbb{C}$ be a domain.
	A domain $D_0\subset D$ is called a \emph{hole} of $D$
	if $D_0$ is simply connected and $\overline{D_0}\subset D$.
	By definition, $D$ is called a domain with disjoint holes
	$(D_0^i)_{i=1}^l$ if $D_0^i$ is a hole of $D$ for all $i=1,\dots, l$ and the sets $\overline{D_0^i}, i=1, \dots, l$ are pairwise disjoint.
	For a set $D$ with disjoint holes $(D_0^i)_{i=1}^l$ we
	call $D_+:=D\setminus (\bigcup_{i=1}^l \overline{D_0^i})$
	an \emph{atoll domain}. The holes $(\overline{D_0^i})_{i=1}^l$ are called
	\emph{lagoons} of the atoll domain.  
\end{definition}
A prototypical domain with one hole is an annulus. More precisely, if for $z\in \mathbb{C}$ and $s>r>0$ we denote by $B_{r,s}(z)$ the 
annulus
$$
	B_{r,s}(z):=\{w\in\mathbb{C}:\ r<|w-z|<s\},
$$
then $B_{r,s}(z)$ is an atoll domain with one hole. (We shall use the notation $B_r(z)$ for the open disc with radius $r$ and center $z$.)
 Associated with a domain with holes we define the following 
 class of functions which will act as our model for the functions $F_j$ mentioned
 in Section \ref{sec:new}.
 \begin{definition}\label{def:funcclass}
 Suppose that $D$ is a bounded atoll domain with disjoint lagoons $(D_0^i)_{i=1}^l$ and let
 $\Delta \geq \delta >0$. Then we define 
 the function class $\mathcal{H}(D,(D_0^i)_{i=1}^l,\delta,\Delta)$
 of \emph{atoll functions} associated with $D$ and $(D_0^i)_{i=1}^l$
 as follows:
\begin{multline}\label{funcclass}
	\mathcal{H}(D,(D^i_0)_{i=1}^l,\delta,\Delta):=\\
	\left\{F\in C^1(D):\ 
	\max\{|F(z)|,| \nabla |F|(z)|\} \le \Delta \mbox{ for all } z\in D,\  |F(z)|\geq \delta \mbox{ for all } z\in D_+\right\}.
\end{multline}
\end{definition}
The interpretation of Definition \ref{def:funcclass} is straightforward.
It consists of functions on $D$ which are large on an atoll $D_+$
and possibly small on a number of lagoons $D_0^i$ which are encircled by an atoll $D_+$. See Figure \ref{fig:atoll} for an illustration.

\begin{figure}[h!]
	\centering
	\includegraphics[width = .9\textwidth]{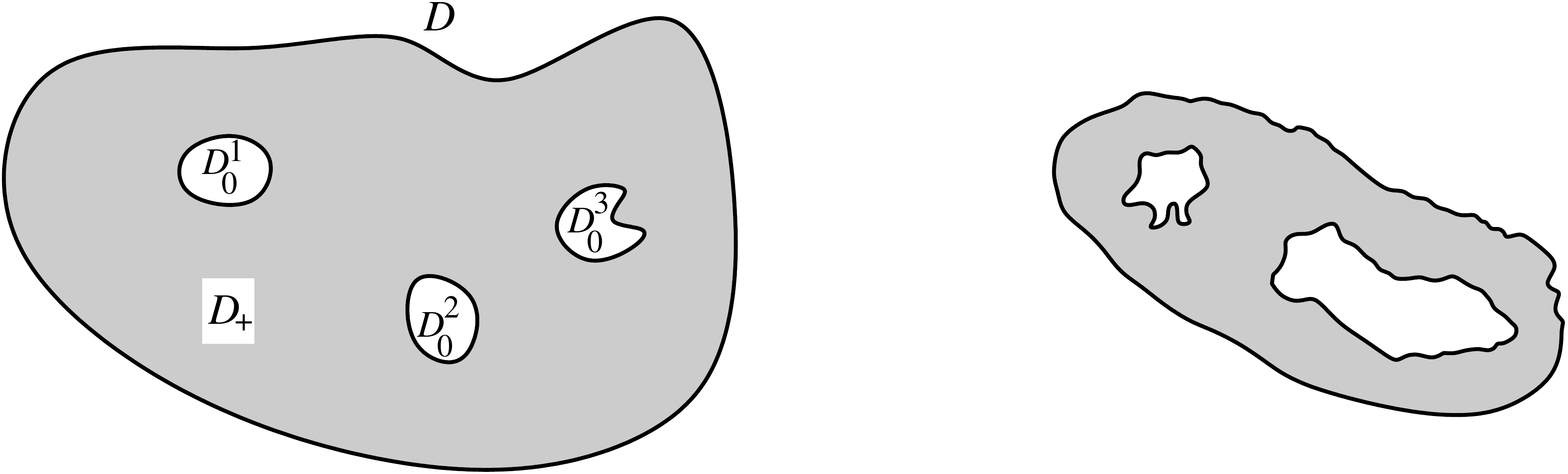}
\caption{Left: An atoll domain with three lagoons: $D$ is the open domain enclosed by the outer curve, $D_0^1$, $D_0^2$ and $D_0^3$ are the three ``holes'' or {\em lagoons}, and $D_+$, the shaded area, is the {\em atoll domain}. Right: although most atoll islands (in their standard geographic meaning) are sickle-shaped, with lagoons in the shape of large bays, narrowly connected with the sea or ocean surrounding the island, some are indeed similar to the figure on the left; given here is the shape of Teeraina island, a coral atoll that is part of Kirabati, at about $4.71^{\circ}$ North latitude and $160.76^{\circ}$ West longitude.}
	\label{fig:atoll}
\end{figure}

The functions we want to consider for phase retrieval (and for which we will show that phase retrieval is uniformly stable) will correspond to a linear combination of atoll functions, each supported on different atolls. Furthermore, as proposed in Section \ref{sec:new},
the reconstruction will be allowed to assign different phases to 
components supported on different atolls.

The present paper establishes such results; as an appetizer we mention the following stability result
which applies to the reconstruction of a function $f\in L^2(\mathbb{R})$
from measuring absolute values of its Gabor transform $V_\varphi f$
as defined in (\ref{eq:wft}), with window function $\varphi(t):=e^{-\pi t^2}$.

Let us suppose that we know a priori that the function $f$
to be recovered can be written as a sum
$f=\sum_{j=1}^k f_j$
with functions $f_j$ each having time-frequency (TF) concentration in an annulus or a disc, i.e.,
\begin{equation}\label{eq:TFcon}
	V_\varphi f_j\in \mathcal{H}(D_j,D_{0,j},\delta_j,\Delta_j),\quad
	\mbox{for } j=1,\dots , k,
\end{equation}
where the $D_j$ are (possibly disjoint) discs $D_j:=B_{s_j}(z_j)$, each with one hole, $D_{0,j}:=B_{r_j}(z_j)$ for $0\le r_j<s_j$ and $z_j\in \mathbb{C}$ for $j=1,\dots , k$.
In audio processing, each of the $f_j$'s may be interpreted as different tones and in different periods of time, each having its TF concentration on the set $D_j$ in the following sense. 
\newtheorem*{theorem*}{Theorem}
\newtheorem*{definition*}{Definition}
\begin{definition}
	For $B\subset \mathbb{R}^2$ and $\varepsilon >0$  
	we say that \emph{$f\in L^2(\mathbb{R})$ is $\varepsilon$-concentrated
	in $B$}, if 
	$$
		\int_{\mathbb{R}^2\setminus B}|V_\varphi f(x,y)|^2dx dy \le 
		\varepsilon^2.
	$$
\end{definition}
We use the notation $W^{1,p}(D)$ for the Sobolev space with norm
$$
	\norm{F}{W^{1,p}(D)}=\norm{F}{L^p(D)}+\norm{\nabla F}{L^p(D)}.
$$
With these definitions and notation, we can now formulate the following theorem that states one of our stability results:
\begin{theorem}\label{thm:gaborapp}
	Suppose that $f=\sum_{j=1}^k f_j\in L^2(\mathbb{R})$ 
	such that (\emph{\ref{eq:TFcon}}) holds true with each $f_j$
	 $\varepsilon_j$-concentrated in $D_j$.
	Suppose that $g\in L^2(\mathbb{R})$ can likewise be written as 
	$g=\sum_{j=1}^k g_j$ with each $g_j$ $\varepsilon_j$-concentrated in $D_j$. Then there exists a continuous function $\rho:[0,1) \to \mathbb{R}_+$ and a uniform constant $c>0$ so that the following estimate
	holds:
	\begin{align*}
		\inf_{\alpha_1,\dots, \alpha_k\in \mathbb{R}} &
			\sum_{j=1}^k \norm{f_j-e^{\i \alpha_j}g_j}{L^2(\mathbb{R})}
			\le c\cdot 
			\sum_{j=1}^k \frac{\Delta_j^2}{\delta_j^2} (1+\rho(r_j/s_j)\cdot s_j)\cdot \\
			& \left(1+(r_j/s_j)^{1/2}\cdot \rho(r_j/s_j)\cdot (s_j+1)\cdot e^{r_j^2 \pi/2}\right)
			\cdot 
			\norm{|V_\varphi f|-|V_\varphi g|
			}{W^{1,2}((D_j)_+)} +\sum_{j=1}^k\varepsilon_j.
	\end{align*}
\end{theorem}
The theorem states that a function that is the sum of components, each of which has a Gabor transform of type (\ref{eq:TFcon}), can be stably reconstructed from the absolute values of its Gabor transform, whenever its Gabor transform is concentrated on a number of atolls with lagoons that are not too large. 

Note that as the lagoons get large, more precisely, if we let $r_j$ grow while keeping the ratios $r_j/s_j$ fixed, the stability of reconstruction
degenerates \emph{at most} exponentially in their area. 
This is completely in line
with the results of \cite{cahill2016phase}, and in particular with the example mentioned in Section \ref{sec:instab} for which the stability of the reconstruction degenerates \emph{at least} exponentially in the size of its corresponding lagoon. Therefore, we believe that such a decay is not a proof artifact but a fundamental barrier to stable phase retrieval, related to the TF-localization properties
of the window $\varphi$, see also Remark \ref{rem:lagoondecay} in Section \ref{sec:cauchy}.

One can construct an example of phase retrieval from Gabor measurements in the spirit of Example \ref{ex:cahill} of real-valued measurements in 1D: In \cite{aifariunstable} some of the authors construct two functions $f_a^+$, $f_a^-$, for which the (Gabor transform) measurements are close to each other in absolute value but such that $\|f_a^+-e^{\i \alpha} f_a^-\|_{L^2(\R)}$ is not small for any phase factor $e^{\i \alpha}$, $\alpha \in \mathbb{R}$. The functions $f_a^+$, $f_a^-$ are constructed such that their Gabor transforms are concentrated on two separated discs $B_{r_0}((-a,0))$ and $B_{r_0}((a,0))$, so that they can be viewed as atoll functions. Applying Theorem \ref{thm:gaborapp} to this example gives stability of the phase retrieval problem with a stability constant that is independent of $a$. In contrast, in the classical sense (i.e. when $V_\varphi f_a^+$, $V_\varphi f_a^-$ are not treated as atoll functions), phase retrieval is unstable in this example with the stability constant deteriorating exponentially in $a^2$. We note however, that the stability constant from Theorem \ref{thm:gaborapp} is not independent of the size of atolls, i.e. of the radius $r_0$. In fact, it grows exponentially in $r_0^2$. Recent work \cite{grohsrathmair} by one of the authors has developed improved results that overcome this growth of the stability constant in the size of the atolls. 

Theorem \ref{thm:gaborapp} is a special case of our much more general Theorem \ref{prop:main}, proved 
in Section \ref{sec:Gabor} below. Theorem \ref{prop:main} however applies to 
a much wider class of measurement scenarios. Another application, discussed in Section \ref{sec:cauchy}, concerns
the phase retrieval problem from measuring absolute values of the Cauchy wavelet
transform of a signal.

\subsection{Proof Strategy}
We briefly describe the underlying mechanism in the proof of the above-mentioned stability theorem. 

\begin{itemize}
	\item At the backbone of Theorem \ref{thm:gaborapp} lies the well-known
	fact that the Gabor transform $F(x,y):=
	V_\varphi f(x,-y)$ is a holomorphic function, up to normalization.
	More precisely, there exists a function $\eta$ such that
	the product $\eta\cdot F$ is holomorphic, see Theorem \ref{thm:gaboranal}.
	In fact, in Theorem \ref{prop:main}, we establish a general stability result for atoll functions
	which are, up to normalization, holomorphic.
	
	\item A key insight leading to this result is the observation that, for a holomorphic function $F$
	the rate of change of $F$ is dominated by the rate of change of $|F|$. This
	fact, which is Lemma \ref{keylemma}, follows directly from the Cauchy-Riemann equations. 
	\item Lemma \ref{keylemma} then allows us to prove a stability result for atoll 
	functions, restricted to the atoll $D_+$ on which a lower bound on 
	their absolute value holds true.
	
	\item In order to also establish a stability bound on the lagoons $(D_0^i)_{i=1}^l$ we use a version of the maximum principle and a trace theorem 
	for Sobolev functions to prove that the reconstruction error on the lagoons
    $(D_0^i)_{i=1}^l$ can be dominated by the approximation error on the atoll $D_+$
    which has been controlled in the previous step. These two steps are carried out
    in Section \ref{sec:proof}. The proof turns out to be involved and dependent on 
    a number of preparatory results which are summarized in Section \ref{sec:prep}.
\end{itemize}

Our main result is Theorem \ref{prop:main} which establishes a stability result
for arbitrary atoll functions that arise from holomorphic measurements (up to normalization). Theorem \ref{thm:gaborapp} then comes as a corollary, but 
the machinery of Theorem \ref{prop:main} allows to deduce stability of phase retrieval for any type of measurements which depend holomorphically on its
parameters. As a further example we mention Cauchy wavelets which have
been treated previously in \cite{waldspurger2015these}.

\subsection{Outline}
The article is structured as follows. Section \ref{sec:prep} provides a package of all the preparatory tools that will be needed later. In particular, we describe analytic Poincar\'e inequalities and the relation of the analytic Poincar\'e constant to the classical Poincar\'e constant in Section \ref{subsec:poincare}. Sections \ref{subsec:trace} and \ref{subsec:bv} outline the results that are needed to control the reconstruction error on the lagoons $(D_0^i)_{i=1}^l$. Stable point evaluations and the simultaneous control of two different constants that will appear in the main result of this paper are treated in Section \ref{subsec:pointev}.

Section \ref{sec:main} features our main result (Theorem \ref{prop:main}) and gives its illustration for two concrete examples: the case of the domain $D=D_+$ being a disc (Section \ref{sec:disc}) and the case of $D_+$ being an annulus (Section \ref{sec:annulus}). In the remainder of this section, the cases of magnitude measurements of the Gabor transform (Section \ref{sec:Gabor}) and of the Cauchy wavelet transform (Section \ref{subsec:cauchy}) are studied and the stability constants are quantified. We give the proof of the main theorem (Theorem \ref{prop:main}) in Section \ref{sec:proof}.

\section{Preparatory Results}\label{sec:prep}
In the course of our work we 
will use several auxiliary results that are summarized in this section. For an overview of the main results of this paper, the reader may want to visit Section \ref{sec:main} directly.
We consider a path-connected domain $D\subset \mathbb{C}$ which is sufficiently nice (e.g. Lipschitz domain) and let $\mathcal{O}(D)$ denote the space of holomorphic functions from $D$ to $\mathbb{C}$. 

We will always write $z=x+\i y\in \mathbb{C}$ and $F(z)=u(x,y)+\i v(x,y)$.
We denote $F'(z)=u_x(x,y)+\i v_x(x,y)$, and $\nabla F(z) = (\nabla u(x,y), \nabla v(x,y))\in \mathbb{R}^{2\times 2}$.

Any $F\in \mathcal{O}(D)$ satisfies the Cauchy-Riemann equations
\begin{equation}\label{eq:CR}
	u_x = v_y \quad \mbox{and}\quad u_y = -v_x.
\end{equation}
A key object of our study is the absolute value 
$|F|:D\to \mathbb{R}$ and its gradient $\nabla |F|=(|F|_x,|F|_y)^T$.
For a subset $B\subset \mathbb{C}$ we denote by $|B|$ its area and by $\chi_B$ its indicator function. 
We write 
$$
	\mathbb{R}_+=\{x\in \mathbb{R}:\  x>0\}
$$
and
$$
	\mathbb{C}_+=\{z = x+\i y:\ x\in \mathbb{R},\ y \in \mathbb{R}_+\}.
$$
 \subsection{Analytic Poincar\'e Inequalities}\label{subsec:poincare}
 We shall rely several times on the validity of
an analytic Poincar\'e inequality. A domain $D$ is said to be an analytic $p$-Poincar\'e domain if for $z_0 \in D$, there exists a constant $C^a_{poinc}(p,D,z_0)>0$ such that
\begin{equation}\label{eq:Poincare}
	\norm{F - F(z_0)}{L^p(D)} \le C^a_{poinc}(p,D,z_0)\norm{F'}{L^p(D)}
\end{equation}
for all $F\in \mathcal{O}(D)$, and $1\le p\le \infty$.

Such inequalities are studied in \cite{poincare}. Although (\ref{eq:Poincare}) features the point $z_0 \in D$, it turns out that whether or not the domain $D$ is an analytic Poincar\'e domain is independent of $z_0$. However \cite{poincare}, the best-possible constant $C^a_{poinc}(p,D,z_0)$ depends on the choice of $z_0$. 
Denote by $C_{poinc}(p,D)$ the usual Poincar\'e constant
of the domain $D$, i.e. the optimal constant $C$ such that
$$
	\norm{F - F_D}{L^p(D)} \le C\norm{\nabla F}{L^p(D)},
$$
where $F_D:=\frac{1}{|D|}\int_D F(z)dz$. Then we have the following estimate for $C^a_{poinc}(p,D,z_0)$:
\begin{lemma}\label{lem:analpoincconst}
	$$
		C^a_{poinc}(p,D,z_0)\le C_{poinc}(p,D)\cdot \left(1+\left(\frac{|D|}{\pi \dist(z_0,\partial D)^2}\right)^{1/p}\right).
	$$
\end{lemma}
\begin{proof}
	In \cite[p. 365]{poincare} the case $p=1$ is shown; 
	the general case can be done analogously.
	Let $r:=\dist(z_0,\partial D)$ and consider the ball $B=B_r(z_0)$.
	By the mean-value property it holds that $F(z_0) = F_B$.
	Therefore we have
	$$
		\norm{F - F(z_0)}{L^p(D)}= \norm{F - F_B}{L^p(D)}.
	$$
	With this, the triangle inequality yields
	\begin{eqnarray*}
		\norm{F - F(z_0)}{L^p(D)} & \le & \norm{F - F_D}{L^p(D)}
		+\norm{F_D - F_B}{L^p(D)}.
	\end{eqnarray*}
	Now we observe that
	$$
		|F_B - F_D| \le \frac{1}{|B|}\int_B |F(z) - F_D|dz 
		\le \frac{1}{|B|}|B|^{1-1/p}\norm{F-F_D}{L^p(B)},
	$$
	where the last inequality follows from H\"older's inequality. 
	Now it remains to observe that $\|F-F_D\|_{L^p(B)}\le \|F-F_D\|_{L^p(D)}$
	and $|B|=\pi \dist(z_0,\partial D)^2$ to arrive at the desired result.
\end{proof}
Essentially, Lemma \ref{lem:analpoincconst} states that whenever
$z_0$ lies in a central location of $D$ (i.e. not too close to $\partial D$), the constant 
$C^a_{poinc}(p,D,z_0)$
can be controlled by the classical Poincar\'e constant $C_{poinc}(p,D)$ which is well-studied. For instance the following result 
is known \cite{payne1960optimal}.
\begin{theorem}\label{thm:poincconv}
	Suppose that $D\subset \mathbb{C}$ is a bounded, convex domain 
	with Lipschitz boundary. Then
	$$
		C_{poinc}(2,D)\le \frac{\mbox{diam}(D)}{\pi}.
	$$
\end{theorem}
For non-convex domains the determination of the optimal Poincar\'e constant is more difficult. 
For the annulus $B_{r,s}(z)$ the following result is known.
\begin{theorem}\label{thm:poincann}
	Suppose that $D= B_{r,s}(z)$. Then there exists a uniform constant $c>0$ such that
	$$
		C_{poinc}(2,D)\le c\cdot s.
	$$
\end{theorem}
\begin{proof}
 	By a scaling argument it is easily seen that
 	$$
 		C_{poinc}(2,B_{r,s}(z)) = s \cdot C_{poinc}(2,B_{r/s,1}(0)).
 	$$
 	The function $h: \tau \mapsto C_{poinc}(2,B_{\tau,1}(0))$ is continuous on 
 	$(0,1)$ because the Poincar\'e constant depends continuously on the domain
 	\cite{grebenkov2013geometrical}. In \cite{gottlieb1985eigenvalues} it is shown 
 	that the function $h$ extends continuously to the endpoint $\tau = 1$ and in 
 	\cite{hempel2006lowest} it is shown that the function $h$ extends continuously to the endpoint
 	$\tau = 0$. Therefore, $h$ is continuous on the closed interval $[0,1]$, and hence bounded, 
 	which proves the statement.
\end{proof}
For more general domains which arise as a diffeomorphic image of a convex domain or an annulus, one can obtain estimates on the Poincar\'e constant
by studying the Jacobian of the diffeomorphism but in the present paper
we are content with knowing the Poincar\'e constant on convex domains and on annuli.
\subsection{Sobolev Trace Inequalities}\label{subsec:trace}
In what follows, we will consider inequalities involving the $L^p$-norm of functions on the piecewise smooth boundary of a bounded domain $D \subset \mathbb{C}$. We define it as
$$
	\|F\|_{L^p(\partial D)}:= \Big( \int_a^b |F(\gamma(t))|^p \, |\gamma'(t)| dt \Big)^{1/p},
$$
where $\gamma:[a,b) \to \partial D$ can be any bijective parameterization of $\partial D$.

The Sobolev trace inequality \cite{evans} provides an upper bound for this norm, which will be important for our purposes:
\begin{theorem}\label{thm:trace}
	Suppose that $D\subset \mathbb{C}$ is a bounded domain with Lipschitz boundary $\partial D$. Then there exists a constant $C_{trace}(p,D)$
	with 
	$$
		\norm{F }{ L^p(\partial D)}\le C_{trace}(p,D)\norm{F}{W^{1,p}(D)}.
	$$
\end{theorem}
The next result provides concrete estimates of the trace constant for discs and annuli. It says that the trace constant behaves nicely for annuli that are not too thin. 
\begin{theorem}\label{thm:tracann}
	There exists a continuous function $\rho:[0,1)\to \mathbb{R}$ 
	with $\lim_{\tau \to 1_-}\rho(\tau) = \infty$ such that 	%
	\begin{align*}
		C_{trace}(2,B_{r,s}(z))&\le \rho(r/s)\cdot (s^{1/2}+ s^{-1/2}).
	\end{align*}
\end{theorem}
\begin{proof}
By a scaling argument, one can verify that
\begin{align*}
	C_{trace}(2,B_{r,s}(z)) &\leq s^{1/2} \cdot C_{trace}(2,B_{r/s,1}(z)), \quad \mbox{for } s \geq 1, \\
	C_{trace}(2,B_{r,s}(z)) &\leq s^{-1/2} \cdot C_{trace}(2,B_{r/s,1}(z)), \quad \mbox{for } s < 1.
\end{align*}
The statement then follows by noting that  $ C_{trace}(2,B_{\tau,1}(z))<
\infty$ for $\tau \in [0,1)$. 
\end{proof}
\subsection{Boundary Values of Holomorphic Functions}\label{subsec:bv}
Another key fact we shall use is that the $L^p$-norm of a holomorphic
function on a simply connected domain is dominated by its $L^p$-norm
on the boundary. 
\begin{theorem}\label{thm:boundvals}
	Suppose that $F\in \mathcal{O}(D)$, where $D\subset \mathbb{C}$
	is a bounded and simply connected domain with smooth boundary. Then there
	exists a constant $C_{bound}(p,D)>0$ such that 
	$$
		\norm{F}{L^p(D)} \le C_{bound}(p,D)\norm{F}{L^p(\partial D)}
	$$
	for all bounded functions $F\in \mathcal{O}(D)$.
\end{theorem}
\begin{proof}
	We assume without loss of generality that $D=B_1(0)$. The general 
	case can be handled using the Riemann mapping theorem. 
	We shall make use of the Hardy space $H^p$, 
	consisting of all functions $F\in \mathcal{O}(B_1(0))$ with
	finite $H^p$-norm, defined by
	$$
		\norm{F}{H^p}^p:=\sup_{0<r<1}\frac{1}{2\pi}\int_0^{2\pi}
		|F(r \cdot e^{\i\varphi})|^p d\varphi.
	$$
	It is well-known (see for instance \cite{katznelson})
	that any $F\in H^p$ can be extended to the boundary $\partial B_1(0)$
	and that
	\begin{equation}\label{eq:hardy}
		2\pi \|F\|_{H^p}^p=\|F\|_{L^p(\partial D)}^p.
	\end{equation}
	We further note that 
	$$
		\frac{1}{2\pi}\norm{F}{L^p(B_1(0))}^p
		= \frac{1}{2\pi}\int_0^1\int_0^{2\pi}
		|F(r \cdot e^{\i\varphi})|^p rdrd\varphi
		\le \norm{F}{H^p}^p.
	$$
	Combining this result with (\ref{eq:hardy}) yields the desired result.
\end{proof}
For discs $B_r(z)$ a simple scaling argument leads to the following result.
\begin{theorem}\label{thm:bounddisc}
	For all 
	$r>0$, $z\in \mathbb{C}$ and
	 $D=B_r(z)$ we have $C_{bound}(p,D)\le r^{1/p}$.
\end{theorem}
For more general simply connected domains the constant $C_{bound}(p,D)$
depends on upper and lower bounds on the Jacobian of the Riemann mapping 
from $D$ to $B_1(0)$. 

\subsection{Stable Point Evaluations}\label{subsec:pointev}
Given a function $G\in L^p(D)$, the proof of our main result will
require us to pick a point $z\in D$ with a small sampling constant
which is defined as follows.
\begin{definition}\label{def: Csamp}
	Let $D$ be a domain and $G\in L^p(D)$. Then we define, for 
	$z_0\in D$ and $1\le p\le \infty$ the \emph{sampling constant}
	$$
		C_{samp}(p,D,z_0,G):=
		\inf \{C>0:\ \norm{G(z_0)}{L^p(D)}\le C\norm{G}{L^p(D)}\}.
	$$
\end{definition}

To control the constant $C(z_0,p,D_+,(D_0^i)_{i=1}^l)$ in our main result Theorem \ref{prop:main},  it is necessary to control $C_{samp}(p,D_+,z_0,|F_2|-|F_1|)$ and $C^a_{poinc}(p,D_+,z_0)$ simultaneously. 

The purpose of this subsection is to show that this can indeed be achieved for general domains $D$ and functions  $G\in L^p(D)$.

We start with the following lemma which shows that there exist ``many'' points with a given sampling constant.
\begin{lemma}\label{lem:z0}
	Suppose that $D\subset \mathbb{C}$ is a domain
	and let $G\in L^p(D)$ for $1\le p\le \infty$. For $C>1$ 
	we denote 
	$$
		D_C(G):=\left\{z_0\in D:\ |G(z_0)| |D|^{1/p}\le C\norm{G}{L^p(D)}\right\}. 
	$$
	Then
	$$
		|D_C(G)|\geq |D|\cdot \left(1-\frac{1}{C^p}\right).
	$$
\end{lemma}
\begin{proof}
	We compute
	$$
		\int_{D\setminus D_C(G)}|G(x)|^p dx
		+ \int_{D_C(G)}|G(x)|^p dx
		= \norm{G}{L^p(D)}^p 
	$$
	By the definition of $D_C(G)$ we have that
	$$
		|G(x)|^p> \frac{C^p}{|D|}\norm{G}{L^p(D)}^p \mbox{  for all }x\in D\setminus D_C(G),
	$$
	and this implies that
	$$
		|D\setminus D_C(G)| \frac{C^p}{|D|} \norm{G}{L^p(D)}^p 
		+\int_{D_C(G)}|G(x)|^p dx \le \norm{G}{L^p(D)}^p.
	$$
	Consequently,
	$$
		(|D|-|D_C(G)|) \frac{C^p}{|D|}  
		\leq 1
	$$
	and this yields the statement.
\end{proof}
Lemma \ref{lem:z0} implies that if we define  $C(t):=\frac{1}{(1-t)^{1/p}}$, then for any $0<t<1$ and $G\in L^p(D)$
we have
$$
	|D_{C(t)}(G)| \geq t|D|.
$$
Next, we define (cf. Figure \ref{fig:s_t})
\begin{equation}\label{eq:stdef}
	s_t(D):=\inf_{S\subset D,\ |S|=t|D|} \sup_{z\in S} \dist(z,\partial D).
\end{equation}
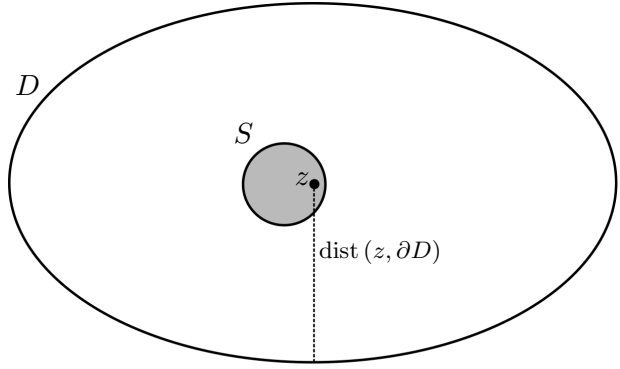
\begin{figure}[h!]
	\centering
	\begin{overpic}[width = .5\textwidth]{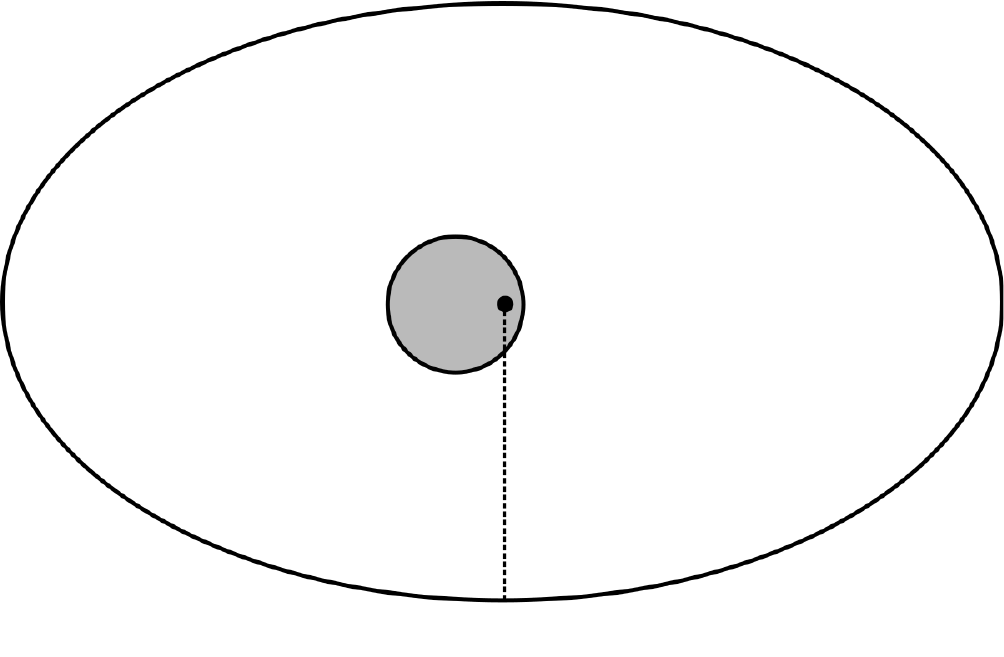}
	 \put(47,35){$z$}
	\put(37,42){$S$}
	\put(51,23){\footnotesize $\dist(z,\partial D)$}
	\put(1,50){$D$}
\end{overpic}
	\caption{A domain $D$, a subset $S$ and an element $z \in S$ that maximizes $\dist(z,\partial D)$. The value of $s_t(D)$ is then obtained by taking the infimum of these quantities over all $S\subset D$ with the same area $|S|=t|D|$, i.e., $s_t(D)=\inf\limits_{S\subset D,\ |S|=t|D|} \sup\limits_{z\in S} \dist(z,\partial D)$.}
	\label{fig:s_t}
\end{figure}
For ``nice'' domains, the quantity $s_t(D)$ can be controlled easily. We mention 
the following result; the proof is an elementary calculus computation.
\begin{lemma}\label{lem:shape}
	For all $s>r>0$ and $z\in \mathbb{C}$ we have the estimate
	$$
		s_{1/2}(B_r(z)) \geq c\cdot r\quad \mbox{and}\quad
		s_{1/2}(B_{r,s}(z))\geq c\cdot (s-r),
	$$
	with $c=1-\frac{1}{\sqrt{2}}$.
\end{lemma}
\begin{figure}[h!]
	\centering
	\begin{overpic}[width = .3\textwidth]{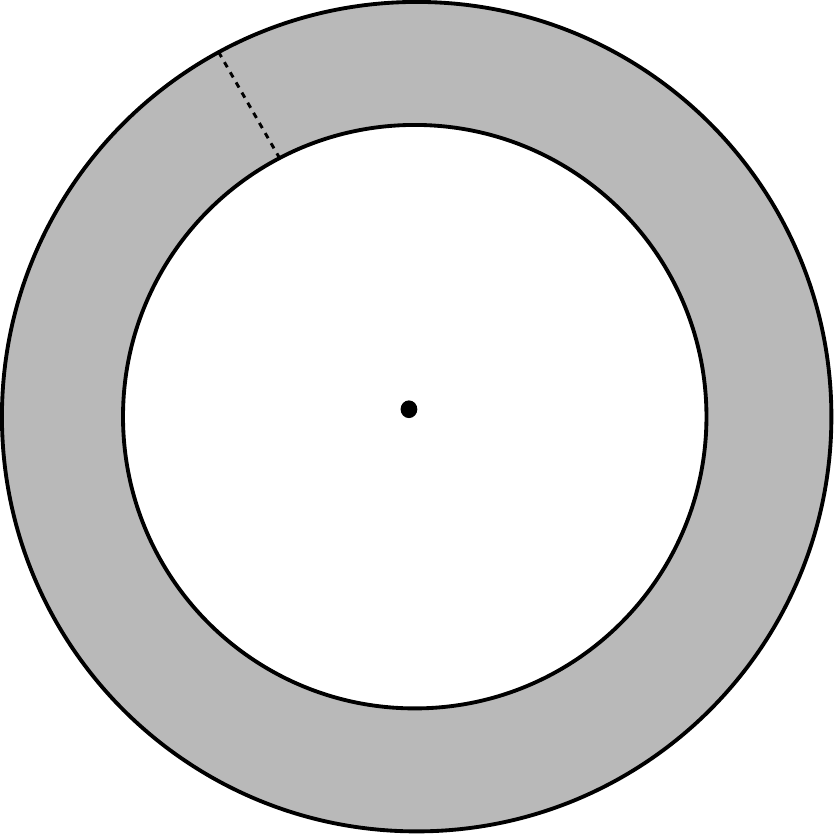}
	 \put(48,46){$z$}
	\put(86,60){$S$}
	\put(31,88){\footnotesize $s_{1/2}(D)$}
	\put(72,96){$D=B_r(z)$}
\end{overpic}
	\caption{For $D=B_r(z)$, $s_{1/2}B_r(z)$ is attained by the subset $S=B_{r_0,r}(z)$ with $r_0=\frac{1}{\sqrt{2}}r$.}
	\label{fig:s_t-disc}
\end{figure}
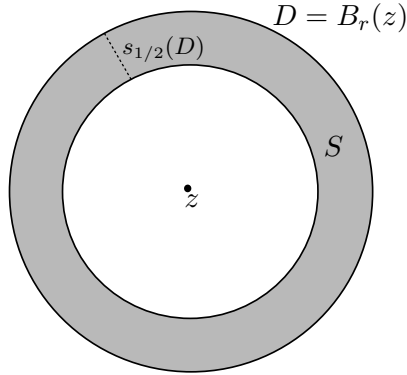

Control of $s_t$ lets us gain control over both the sampling constant and 
the analytic Poincar\'e constant.
As an immediate consequence of Lemma \ref{lem:z0} we have the following result.
\begin{lemma}\label{lem:tlem}
	Let $0<t<1$, let $ D\subset \mathbb{C}$ be a domain and $G\in L^p(D)$. 
	There exists $z_0\in D$ with 
	$$
		C_{samp}(p,D,z_0,G)\le \frac{1}{(1-t)^{1/p}},
	$$
	and 
	$$
		C^a_{poinc}(p,D,z_0)\le C_{poinc}(p,D)\left(1+\left(\frac{|D|}{\pi s_t(D)^2}\right)^{1/p}\right).
	$$
	\end{lemma}
\begin{proof}
	Picking $C(t)= \frac{1}{(1-t)^{1/p}}$
	we get that $|D_{C(t)}(G)|\geq t|D|$ by Lemma \ref{lem:z0}.
	Therefore 
	$$
		\sup_{z\in D_{C(t)}(G)} \dist(z,\partial D)\geq s_t(D)
	$$
	and thus there exists $z_0\in D_{C(t)}(G)$ with 
	$$
		\dist(z_0,\partial D)\geq s_t(D).
	$$
	Lemma \ref{lem:analpoincconst} now immediately implies the
	claimed bound for $C^a_{poinc}(p,D,z_0)$.
	
	On the other hand, by the definition of $C(t)$ and the fact that 
	$z_0\in D_{C(t)}(G)$, we get the desired bound on the sampling constant
	which proves the statement.
\end{proof}
In order to make use of Lemma \ref{lem:tlem} to estimate the 
constants $C_{samp}(p,D,z_0,G)$ and \\$C^a_{poinc}(p,D,z_0)$ we need
to control only the quantity $s_t(D)$.
For ``nice'' domains $D$ we expect that $s_t(D)$ behaves like 
the diameter $\mathrm{diam}(D)$ and also that $\mathrm{diam}(D)^2$
behaves like $|D|$; hence the quotient $\frac{|D|}{\pi s_t(D)^2}$ would be uniformly bounded which implies that, for a suitable choice
of $z_0\in D$, the constant $C^a_{poinc}(p,D,z_0)$ is comparable to the classical Poincar\'e constant $C_{poinc}(p,D)$, while 
$C_{samp}(p,D,z_0,G)$ is bounded by a fixed constant. These considerations will give us full control 
of all underlying constants for sufficiently nice domains, needed in the estimates in the next section. 

\section{Stability of Phase Reconstruction from Holomorphic Measurements}\label{sec:main}
The purpose of this section is to formulate the following fundamental result and discuss some of its implications.
\begin{theorem}\label{prop:main}
	 Suppose that $F_1$ belongs to a class
	 of atoll functions as in Definition \ref{def:funcclass}, i.e., $F_1\in \mathcal{H}(D,(D_0^i)_{i=1}^l,\delta,\Delta)$. 
	 Assume further that $F_2\in C^1(D)$ such that
	 there exists a continuous function
	 $\eta: D \to \mathbb{C}$ for which
	 both functions $\eta\cdot F_1,\ \eta\cdot F_2\in \mathcal{O}(D)$. 
	 Suppose that $1\le p\le \infty$.
	 
	Pick $z_0\in D_+$. We denote $C_{samp}:=C_{samp}(p,D_+,z_0,|F_1|-|F_2|)$ meaning that
	 \begin{equation}\label{eq:z0est}
	 	 \||F_1(z_0)|-|F_2(z_0)|\|_{L^p(D_+)}\le C_{samp}\cdot \||F_1|-|F_2|\|_{L^p(D_+)}.
	 \end{equation}

	 Then the following estimate holds:
	 \begin{equation}\label{eq:mainest}
	 	 \inf_{\alpha \in \mathbb{R}}\norm{F_1-e^{\i\alpha}F_2}{L^p(D)}
		 \le C(z_0,p,D_+,(D_0^i)_{i=1}^l)\frac{\Delta^2}{\delta^2}\norm{|F_1|-|F_2| }{W^{1,p}(D_+)},
	 \end{equation}
	 where for the constant $C(z_0,p,D_+,(D_0^i)_{i=1}^l)$ we may choose (with a suitably large but uniform constant $c>0$):
	 \begin{multline}\label{eq:const}
	 	C(z_0,p,D_+,(D_0^i)_{i=1}^l)=
	 	 c\cdot (C^a_{poinc}(D_+)+ C_{samp}
	 	 	 	 \\ + 
	 	 	 	 \sum_{i=1}^lC_{bound}(D_0^i)\cdot \var(\eta,D_0^i)\cdot 
	 	 	 	 C_{trace}(D_+)
	 	 	 	 (C^a_{poinc}(D_+)+C_{samp})),
	 \end{multline}
	 where we have omitted the dependence of the various constants on $p,z_0$ and denote 
	 $$
	 	\var(\eta,D_0^i):=\frac{\max_{z\in \partial D_0^i}|\eta(z)|}{\min_{z\in D_0^i}|\eta(z)|},\quad i=1,\dots ,l.
	 $$
\end{theorem}
\begin{remark}\label{rem:ccontrol}
By Lemma \ref{lem:tlem}, the two constants $C_{samp}$ and $C^a_{poinc}(D_+)$ depending on $z_0$ can be controlled simultaneously. To achieve the best possible $C(z_0,p,D_+(D_0^i)_{i=1}^l)$, $z_0$ should be picked s.t. $\dist(z_0,\partial D_+)$ is large and $\| |F_1(z_0)| - |F_2(z_0)| \|$ is small. 
\end{remark} 

\begin{remark}
In Theorem \ref{prop:main}, we assume that there exists a normalization function $\eta$,  s.t. $\eta\cdot F_1, \eta \cdot F_2\in\mathcal{O}(D)$. In Sections \ref{sec:Gabor} and \ref{sec:cauchy}, we show for $F$ in the image domain of the Gabor or Cauchy wavelet transform, respectively, the existence of explicit functions $\eta$ such that $\eta\cdot F$ is holomorphic on the entire parameter domain. On the other hand, for more general measurements, such global $\eta$ may not exist and for $F\in\mathcal{H}(D, D_0,\delta,\Delta)$, there might be accumulated zeros in $D_0$. In this case, if the accumulated zero set $D_O :=\overline{\{z; \ F_1(z)F_2(z) = 0\}^\circ}\subset D_0$ is simply connected with smooth boundary, then the bound \eqref{eq:mainest} in Theorem \ref{prop:main} still holds with the domain of the $L^p$-norm on the right hand side changing from $D_+$ to $D$ \footnote{This extension requires a generalized version of Theorem \ref{thm:boundvals} for the annulus, which can be shown following the same idea of proof of the disc case but considering the Hardy space defined on an annulus instead, see Theorem 3 in \cite{wang1983real}}.
\end{remark}

Before we provide the lengthy proof of Theorem \ref{prop:main}
in Section \ref{sec:proof},
we pause and provide some special examples which might be illuminating.
To give two simple examples,
in Section \ref{sec:disc} we shall see how to gain 
explicit estimates for the quantity $C(z_0,p,D_+,D_0)$ for $D=D_+$ a disc (i.e., $D_0=\emptyset$)
and in Section \ref{sec:annulus} for $D_+$ an annulus.

These examples should make clear that similar results also hold for more general domains.
\subsection{Example I: A Disc}\label{sec:disc}
In this subsection we shall treat the case $D=D_+=B_r(z)$ and $D_0=\emptyset$. The class 
$\mathcal{H}(D_+,D_0,\delta,\Delta)$ now consists of functions
which are bounded from below by $\delta$ and which (together with their gradient) are bounded 
from above by $\Delta$ on all of $B_r(z)$.
We have the following result.
\begin{theorem}\label{thm:stabdisc}
	 Suppose that $F_1\in \mathcal{H}(B_r(z),\emptyset,\delta,\Delta)$ for some $r>0$ and $z \in \mathbb{C}$. 
	We further assume that $F_2\in C^1(B_r(z))$ such that
	 there exists a continuous function
	 $\eta: B_r(z) \to \mathbb{C}$ for which
	 both functions $\eta\cdot F_1,\ \eta\cdot F_2\in \mathcal{O}(B_r(z))$. 

	 Then there exists a uniform constant $c>0$ such that the following estimate holds.
	 \begin{equation}\label{eq:mainestdisc}
	 	 \inf_{\alpha \in \mathbb{R}}\norm{F_1-e^{\i\alpha}F_2}{L^2(B_r(z))}
		 \le c\cdot (1+r)\cdot \frac{\Delta^2}{\delta^2}\cdot \norm{|F_1|-|F_2| }{W^{1,2}(B_r(z))}.
	 \end{equation}
\end{theorem}
\begin{proof}
We let uniform constants $c$ vary from line to line.
	First we note that, by Lemma \ref{lem:shape}, there exists a uniform constant $c>0$ such that $s_{1/2}(B_r(z))\geq c\cdot r$ with $s_t(B_r(z))$ defined as in (\ref{eq:stdef}).
	It follows from Lemma \ref{lem:tlem} that there exists
	a uniform constant $c>0$ and  $z_0\in B_r(z)$ 
	with 
	$$
		C_{samp}(p,D_+,z_0,G)\le c\quad \mbox{and}
		\quad C^a_{poinc}(p,B_r(z),z_0)\le c\cdot C_{poinc}(p,B_r(z)),
	$$
	where we have put $G:=|F_2|-|F_1|$.
	
	Now it remains to employ Theorem \ref{thm:poincconv} to get a suitable estimate
	on the quantity (\ref{eq:const}) for $p=2$ which, together with Theorem \ref{prop:main} yields the desired result.
\end{proof}
More general results can be obtained for domains $D$ which are diffeomorphic to
$B_r(z)$ in an obvious way. The resulting bounds will depend
on upper and lower bounds of the Jacobian of the mapping which maps $D$ 
to $B_r(z)$.
 
A similar result can also be established for general convex domains $D$ where $r$ in the theorem above may be replaced by $\mbox{diam}(D)$ and the constant $c$ may depend on the geometry of $D$.

We omit the details.
\subsection{Example II: An Annulus}\label{sec:annulus}
To make the general result of Theorem \ref{prop:main} more accessible and to give an idea of the quantitative nature of the stability constant
$C(z_0,p,D_+,D_0)$ we treat here the case of an annulus $D_+=B_{r,s}(z)$
and $D_0=B_{r}(z)$ for $s>r>0$ and some $z \in \mathbb{C}$.
It is interesting to observe the dependence of the stability constant on
the size of the ``lagoon'' $D_0$ on which the phaseless measurements
are allowed to be arbitrarily small. 
We have the following result.
\begin{theorem}\label{thm:stabann}
	 Suppose that $F_1\in \mathcal{H}(B_{s}(z),B_r(z),\delta,\Delta)$ for $s>r>0$. Furthermore, let $F_2\in C^1(B_{s}(z))$ be such that there exists a continuous function
	 $\eta: B_s(z) \to \mathbb{C}$ for which
	 both functions $\eta\cdot F_1,\ \eta\cdot F_2\in \mathcal{O}(B_{s}(z))$. 

	 Then there exist a continuous function $\rho:[0,1)\to\mathbb{R}_+$ 
	 with $\lim_{\rho\to 1_-}=\infty$ and a uniform constant $c > 0$ such that the following estimate holds.
	 \begin{multline}\label{eq:mainestann}
	 	 \inf_{\alpha \in \mathbb{R}}\norm{F_1-e^{\i\alpha}F_2}{L^2(B_s(z))}
		 \le \\
		 c\cdot(1+\rho(r/s)\cdot s)\cdot \left(1+r^{1/2}\cdot \rho(r/s)\cdot (s_j^{1/2}+s_j^{-1/2})\cdot \var(\eta,B_r(z))\right)\cdot \frac{\Delta^2}{\delta^2}\norm{|F_1|-|F_2| }{W^{1,2}(B_{r,s}(z))}.
	 \end{multline}
\end{theorem}
\begin{proof}
	We first observe the elementary fact that $D_+ = B_{r,s}(z)$ and that, by Lemma
	\ref{lem:shape},
	there exists a uniform constant $c>0$ with
	$$
		s_{1/2}(B_{r,s}(z))\geq c(s-r).
	$$
	Using Lemma \ref{lem:tlem} and setting $G:=|F_1|-|F_2|$ this implies the existence
	of $z_0\in B_{r,s}(z)$ and a uniform constant $c$ with
	$$
		C_{samp}(p,D_+,z_0,G)\le c\quad \mbox{and}
		\quad C^a_{poinc}(p,B_{r,s}(z),z_0)\le c\cdot \frac{1}{(1-r/s)^{1/p}} C_{poinc}(p,B_{r,s}(z)).
	$$
	All further constants may be estimated from Theorems
	\ref{thm:poincann}, \ref{thm:bounddisc} and \ref{thm:tracann}
	which, together with Theorem \ref{prop:main} yield the desired result.
\end{proof}
Theorem \ref{thm:stabann} shows that stability can still be retained, 
even if the function $F_1$ is allowed to be small on a large set.
Again, more general results can be derived for domains which are diffeomorphic to an annulus.
\subsection{Phase Retrieval from Gabor Measurements}\label{sec:Gabor}
For a window $g\in L^2(\mathbb{R})$ define the windowed Fourier transform of $f\in L^2(\mathbb{R})$ as 
\begin{equation}\label{eq:Gabor}
	V_g f(x,y) :=\int_{\mathbb{R}}f(t)\overline{g(t-x)}e^{-2\pi \i t y}dt.
\end{equation}
The Gabor transform is defined as the windowed Fourier transform with window $ \varphi(t):=e^{-\pi t^2}$.
The following result is well-known \cite{grochenig}.
\begin{theorem}\label{thm:gaboranal}
	For $z_0=x_0+\i y_0\in \mathbb{C}$ and with $\eta_{z_0}(z):= e^{\pi(|z-z_0|^2/2- \i \cdot (x+x_0)\cdot (y-y_0))},$ 
	the function 
	$$
		F(z):=V_\varphi f(x,-y) \quad \mbox{where }z = x+\i y
	$$
	satisfies that $\eta_{z_0}\cdot F\in \mathcal{O}(\mathbb{C})$.
\end{theorem}
Now consider the problem of stably reconstructing a function from
the absolute values of its Gabor transform. By Theorem \ref{thm:gaboranal}
we are in a position to apply Theorem \ref{prop:main} directly.
\begin{theorem}\label{thm:gaborstab}
	Suppose that $f\in L^2(\mathbb{R})$.
	Suppose that $V_\varphi f$ is an atoll function associated
	with $D_j:=B_{s_j}(z_j)$ and $D_{0,j}:=B_{r_j}(z_j)$ for $0\le r_j<s_j$ and $z_j\in \mathbb{C}$ for $j=1,\dots , k$, meaning that
	$$
		(V_\varphi f) \Big\vert_{D_j} \in \mathcal{H}(D_j,D_{0,j},\delta_j,\Delta_j) \quad \forall j \in \{1, \dots, k \}.
	$$
	Then there exists a continuous function $\rho:[0,1)\to \mathbb{R}_+$ and a constant $c>0$ so that for all $g\in L^2(\mathbb{R})$ the following estimate
	holds:
	\begin{align*}
		\inf_{\alpha_1,\dots, \alpha_k\in \mathbb{R}}
			& \sum_{j=1}^k \norm{V_\varphi f-e^{\i \alpha_j}V_\varphi g
			}{L^2(D_j)} \le c\cdot \Bigg(\sum_{j=1}^k \frac{\Delta_j^2}{\delta_j^2} (1+\rho(r_j/s_j)\cdot s_j)\cdot \\
			& \left(1+r_j^{1/2}\cdot \rho(r_j/s_j)\cdot (s_j^{1/2}+s_j^{-1/2})\cdot e^{r_j^2 \pi/2}\right)\Bigg)\cdot
			\norm{|V_\varphi f|-|V_\varphi g|
			}{W^{1,2}\left(\bigcup_{j=1}^k (D_j)_+\right)}.
	\end{align*}
\end{theorem}
\begin{proof}
	The proof follows directly from Theorem \ref{thm:stabann} together
	with observing that \\$\var(\eta_{z_j},B_{r_j}(z_j))\le c\cdot e^{r_j^2 \pi/2}$
	for a uniform constant $c>0$. 
\end{proof}
We are now ready to conclude the proof of Theorem \ref{thm:gaborapp}, as announced in 
Section \ref{sec:atoll}.
\begin{proof}[Proof of Theorem \ref{thm:gaborapp}]
	It is well-known that the Gabor transform $V_\varphi:L^2(\mathbb{R})\to L^2(\mathbb{R}^2)$ is an isometry
	on its range, see \cite{grochenig}. 
 By assumption, the functions $f_j,g_j$ are $\varepsilon_j$-concentrated in $D_j$ (see Definition \ref{eq:TFcon}). 
 Therefore, 
 $$
 	\norm{f_j - e^{\i \alpha_j}g_j}{L^2(\mathbb{R})}\le \norm{V_\varphi f - e^{\i\alpha_j}
 	V_\varphi g }{L^2(D_j)} + \varepsilon_j.
 $$
 Now, the statement of Theorem \ref{thm:gaborapp} is a direct consequence
	of Theorem \ref{thm:gaborstab}.
\end{proof}
\subsection{Phase Retrieval from Cauchy Wavelet Measurements}\label{subsec:cauchy}
\label{sec:cauchy}
For $g\in L^2(\mathbb{R})$ define the wavelet transform of $f\in L^2(\mathbb{R})$ as 
\begin{equation}\label{eq:Wav}
	W_g f(x,y) :=\frac{1}{|y|^{1/2}}\int_{\mathbb{R}}f(t)\overline{g((t-x)/y)}dt.
\end{equation}
Define the Cauchy wavelet of order $s\in \mathbb{N}$
via its Fourier transform $\widehat\psi(\omega)=\omega^s e^{-2\pi \omega}\chi_{\omega>0}(\omega)$.
The following result is well-known \cite{ascensi2009model}.
\begin{theorem}\label{thm:cauchyanal}
	For $\eta(z):= |1/y|^{s+1/2}$ and 
	any $f\in L^2(\mathbb{R})$ with $\mbox{supp }\widehat f \subset \mathbb{R}_+$,
	the function 
	$$
		F(z):=W_\psi f(x,y) \quad \mbox{where }z = x+\i y
	$$
	satisfies that $\eta \cdot F\in \mathcal{O}(\mathbb{C}_+)$, where $\mathbb{C}_+:= \{x+\i y; \ y \geq 0 \}$.
\end{theorem}
\begin{proof}
	For the convenience of the reader we provide a proof. 
	It is easy to check that, for $f$ with $\mbox{supp }\widehat f\subset \mathbb{R}_+$,
	the function
	$$
		G(z):= \int_{\mathbb{R}_+}\omega^s \widehat f(\omega) e^{-2\pi y\omega }e^{2\pi \i x \omega} d\omega,\quad
		\mbox{for }z=x+\i y\in \mathbb{C_+}
	$$
	is holomorphic on $\mathbb{C}_+$. In fact, it is the holomorphic extension of the $s$-th
	derivative of $f$, if the former exists.
	
	Now note that 
	$$
		W_\psi f (x,y) = f\ast \psi_{y}(x),
	$$
	where 
	$$
		\psi_{y}(t) = \frac{1}{|y|^{1/2}}\psi(-t/y).
	$$
	The Fourier transform of $\psi_y$ is given
	as 
	$$
		\widehat{\psi_y}(\omega) = |y|^{1/2}\widehat\psi (y\cdot \omega)
		= |y|^{s+1/2}\omega^s e^{-2 \pi y \omega}\chi_{\mathbb{R}_+}(\omega).
	$$
	It follows that
	$$
			G(z) =  |y|^{-s-1/2}\cdot W_\psi f (x,y)
	$$
	which proves the statement.
\end{proof}
Using Theorem \ref{prop:main}, the statement of Theorem \ref{thm:cauchyanal}
immediately implies the following result related to the stability of phase retrieval
from Cauchy wavelet measurements. 
\begin{theorem}\label{thm:cauchystab}
	Suppose that $f\in L^2(\mathbb{R})$ with $\mbox{supp }\widehat f \subset \mathbb{R}_+$.
	Suppose that $W_\psi f$ is an atoll function associated
	with $D_j:=B_{s_j}(z_j)$ and $D_{0,j}:=B_{r_j}(z_j)$ for $0\le r_j<s_j$ and $z_j=x_j+\i y_j \in \mathbb{C}_+$ for $j=1,\dots , k$, meaning that
	$$
		(W_\psi f) \Big\vert_{D_j} \in \mathcal{H}(D_j,D_{0,j},\delta_j,\Delta_j) \quad \forall j \in \{1, \dots, k \}.
	$$
	Then, for $g\in L^2(\mathbb{R})$ arbitrary with $\mbox{supp }\widehat g \subset \mathbb{R}_+$, the following estimate
	holds for a continuous function $\rho:[0,1)\to \mathbb{R}_+$ and a constant $c>0$ that are both uniform.
	\begin{align*}
		&\inf_{\alpha_1,\dots, \alpha_k\in \mathbb{R}}
			\sum_{j=1}^k \norm{W_\psi f-e^{\i \alpha_j}W_\psi g
			}{L^2(D_j)}
			\le  \\
			& \le c\cdot \left(\sum_{j=1}^k \frac{\Delta_j^2}{\delta_j^2} (1+\rho(r_j/s_j)\cdot s_j)\cdot \left(1+r_j^{1/2}\cdot \rho(r_j/s_j)\cdot (s_j^{1/2}+s_j^{-1/2})\cdot \left|\frac{1}{1-r_j/y_j}\right|^{s+1/2}\right)\right)\cdot
			\\
			& \quad \norm{|W_\psi f|-|W_\psi g|
			}{W^{1,2}\left(\bigcup_{j=1}^k (D_j)_+\right)}.
	\end{align*}
\end{theorem}
\begin{proof}
	We have that $\var(\eta,B_{r_j}(z_j))\le c\cdot |1-r_j/y_j|^{-s-1/2}$ for a uniform constant $c>0$, so that the statement is a  direct consequence of Theorem \ref{thm:stabann}. 
\end{proof}
\begin{remark}\label{rem:lagoondecay}
	It is interesting to observe how the stability bounds in Theorem \ref{thm:gaborstab} and \ref{thm:cauchystab} deteriorate as the size 
	of the lagoons grows, that is, as the parameter $r_j$ grows. 
	In the case of Gabor measurements this growth
	is of order $e^{r_j^2 \pi/2}$, while in the case of Cauchy wavelets with $s$ vanishing
	moments, the growth is of order $(\frac{1}{1-r_j/y_j})^{s+1/2}$, becoming worse as the number
	of vanishing moments increases. 
	
	Interpreting these quantities in geometric terms we note
	that the area of a lagoon in the parameter space
	of the Gabor transform is of order $r_j^2 \pi$, that is, 
	the stability decays exponentially in the area of the lagoon.
	
	For the wavelet transform the natural notion of area in the upper
	half-plane
	is given by the Poincar\'e metric, i.e., by
	$$
		\mbox{area}_{\mathbb{C}_+}(B):=\int_B \frac{dx dy}{y^2}
	$$
	and a simple calculation gives
	 $$
		\mbox{area}_{\mathbb{C}_+}(B_{r_j}(z_j)) =  \int_0^{2\pi} \int_0^{r_j} \frac{1}{(y_j+\rho \sin \phi)^2} \rho \, d\rho \, d\phi= 2 \pi \left( \frac{1}{\sqrt{1-r_j^2/y_j^2}} - 1 \right),
	$$
	so that
	$$
	\frac{\pi}{\sqrt{2}} \left( \frac{1}{\sqrt{1-r_j/y_j}} -1 \right) \leq  \mbox{area}_{\mathbb{C}_+}(B_{r_j}(z_j)) \leq 2 \pi \left( \frac{1}{\sqrt{1-r_j/y_j}} -1 \right).
	$$
	
	This shows that the stability
	of the phase retrieval from Cauchy wavelet measurements decays
	only polynomially in the area of the lagoon.
	
	This behavior is most likely related to the
	fact that Gabor systems are much more well-localized in the time-frequency plane
	than Cauchy wavelets and that the localization properties of Cauchy wavelets
	increase as the number $s$ of vanishing moments increases. 
	
	It is known that strong localization properties of the measurement system 
	are an obstruction to stable phase retrieval \cite{bandeira2014saving} and in 
	light of this the stability behavior of Theorems \ref{thm:gaborstab} and \ref{thm:cauchystab} is not really surprising.
\end{remark}
\section{Proof of Theorem \ref{prop:main}} \label{sec:proof}
This section is devoted to prove Theorem \ref{prop:main} which
is the main result of this paper. 
The proof follows several steps and relies on the following key lemma.
\begin{lemma}\label{keylemma}
Suppose that $F\in\mathcal{O}(D)$, then
$$ |F'(z)| = \big| \nabla |F|(x,y) \big|\hspace{3em} \forall\, z = x+iy \in D.$$
\end{lemma}
\begin{proof}
Let $u$ and $v$ denote the real and imaginary part of $F$, respectively, i.e., $F(x,y)=u(x,y)+\i v(x,y).$ Then,
\begin{align*}
\partial_x \left|F\right| &= \partial_x(\sqrt{u^2+v^2})\\
&=\frac{1}{2}\cdot\frac{1}{\sqrt{u^2+v^2}}\cdot(2u\cdot u_x + 2v\cdot v_x)\\
&= \frac{u\cdot u_x + v\cdot v_x}{|F|}.
\end{align*}
Similarly, $$\partial_y|F| = \frac{uu_y+vv_y}{|F|} = \frac{-uv_x+vu_x}{|F|},$$ where the last equality follows from Cauchy-Riemann equations.
Therefore, 
\begin{align*}
\big|\nabla|F|\big|^2 &= (\partial_x|F|)^2 + (\partial_y|F|)^2  = \frac{(uu_x+vv_x)^2 + (-uv_x+vu_x)^2}{|F|^2}\\
&= \frac{(u^2+v^2)(u_x^2+v_x^2)}{|F|^2} = u_x^2 + v_x^2 = |F'(z)|^2.
\end{align*}
\end{proof}
Having Lemma \ref{keylemma} at hand we may now proceed to the proof of Theorem \ref{prop:main}, which we restate here for convenience of the reader.

\begin{customthm}{3.1}
	 Suppose that $F_1$ belongs to a class
	 of atoll functions as in Definition \ref{def:funcclass}, i.e., $F_1\in \mathcal{H}(D,(D_0^i)_{i=1}^l,\delta,\Delta)$. 
	 Assume further that $F_2\in C^1(D)$ such that
	 there exists a continuous function
	 $\eta: D \to \mathbb{C}$ for which
	 both functions $\eta\cdot F_1,\ \eta\cdot F_2\in \mathcal{O}(D)$. 
	 Suppose that $1\le p\le \infty$.
	 
	Pick $z_0\in D_+$. We denote $C_{samp}:=C_{samp}(p,D_+,z_0,|F_1|-|F_2|)$ meaning that
	 \begin{equation}\label{eq:z0est-proof}
	 	 \||F_1(z_0)|-|F_2(z_0)|\|_{L^p(D_+)}\le C_{samp}\cdot \||F_1|-|F_2|\|_{L^p(D_+)}.
	 \end{equation}

	 Then the following estimate holds:
	 \begin{equation}\label{eq:mainest-proof}
	 	 \inf_{\alpha \in \mathbb{R}}\norm{F_1-e^{\i\alpha}F_2}{L^p(D)}
		 \le C(z_0,p,D_+,(D_0^i)_{i=1}^l)\frac{\Delta^2}{\delta^2}\norm{|F_1|-|F_2| }{W^{1,p}(D_+)}.
	 \end{equation}
	\end{customthm} 
	We recall that for the constant $C(z_0,p,D_+,(D_0^i)_{i=1}^l)$ one may choose (with a suitably large but uniform constant $c>0$):
	 \begin{multline}\label{eq:const-proof}
	 	C(z_0,p,D_+,(D_0^i)_{i=1}^l)=
	 	 c\cdot (C^a_{poinc}(D_+)+ C_{samp}
	 	 	 	 \\ + 
	 	 	 	 \sum_{i=1}^lC_{bound}(D_0^i)\cdot \var(\eta,D_0^i)\cdot 
	 	 	 	 C_{trace}(D_+)
	 	 	 	 (C^a_{poinc}(D_+)+C_{samp})),
	 \end{multline}
	 where we have omitted the dependence of the various constants on $p,z_0$ and denote 
	 $$
	 	\var(\eta,D_0^i):=\frac{\max_{z\in \partial D_0^i}|\eta(z)|}{\min_{z\in D_0^i}|\eta(z)|},\quad i=1,\dots ,l.
	 $$

\begin{proof}[Proof of Theorem \ref{prop:main}]
Without loss of generality we let $l=1$ and put $D_0:=D_0^1$ (the general 
case being not more difficult).
We need to bound the quantity
	\begin{equation}\label{eq:firstesta}
		\norm{F_2(z)-e^{\i\alpha}F_1(z)}{L^p(D)}\le 
		\norm{F_2(z)-e^{\i\alpha}F_1(z)}{L^p(D_+)}+\norm{F_2(z)-e^{\i\alpha}F_1(z)}{L^p(D_0)}
	\end{equation}
	for suitable $\alpha \in \mathbb{R}$
	and we will develop separate arguments for the two terms on the RHS of the above.
	
{\bf Step 1. } As a first step we start by developing a basic estimate.
	Consider 
	$$
		F:=F_2/F_1.
	$$
	By assumption we have that $\eta \cdot F_1, \eta \cdot F_2 \in  \mathcal{O}(D)$ and $|F_1(z)| \geq \delta$ for $z \in D_+.$ Consequently, $F\in \mathcal{O}(D_+)$.
	Pick $\alpha \in \mathbb{R}$ such that
	\begin{equation}\label{alphapick}
		|F_2(z_0)-e^{\i\alpha}F_1(z_0)|=||F_2(z_0)|-|F_1(z_0)||.
	\end{equation}
	Now consider for $z\in D$ arbitrary
	\begin{eqnarray}
		|F_2(z)-e^{\i\alpha}F_1(z)| & = & |F_1(z)||F(z)-e^{\i\alpha}|\nonumber \\
		&\le & |F_1(z)|\left(|F(z)-F(z_0)|+|F(z_0)-e^{\i\alpha}|\right) \nonumber \\
		&=&|F_1(z)|\left(|F(z)-F(z_0)|+\frac{1}{|F_1(z_0)|}|F_2(z_0)-e^{\i\alpha}F_1(z_0)|\right) \nonumber
		\\
		&=&|F_1(z)|\left(|F(z)-F(z_0)|+\frac{1}{|F_1(z_0)|}||F_2(z_0)|-|F_1(z_0)||\right) \nonumber
		\\
		&\le &
		\Delta \left(|F(z)-F(z_0)|+\frac{1}{\delta}||F_2(z_0)|-|F_1(z_0)||\right)\label{eq:basicest}
	\end{eqnarray}
	{\bf Step 2. }
 	In this step we focus on the second term of (\ref{eq:firstesta})
 	and show that it can actually be absorbed by an estimate on $D_+$. 
	By the analyticity of $\eta \cdot F_1$ and $\eta \cdot F_2$ on $D$, we can apply Theorem \ref{thm:boundvals} to obtain
	$$
		\|\eta \cdot (F_2(z)-e^{\i\alpha}F_1(z))\|_{L^p(D_0)}
		\le C_{bound}(p,D_0)
		\|\eta \cdot (F_2(z)-e^{\i\alpha}F_1(z))\|_{L^p(\partial D_0)}
	$$
	and therefore we get 
	$$
		\norm{F_2(z)-e^{\i\alpha}F_1(z)}{L^p(D_0)}
		\le C_{bound}(p,D_0)\cdot \var(\eta,D_0)\cdot
		\norm{F_2(z)-e^{\i\alpha}F_1(z)}{L^p(\partial D_0)}.
	$$
	We may now estimate further, using (\ref{eq:basicest}), that
	\begin{multline*}
			\norm{F_2(z)-e^{\i\alpha}F_1(z)}{L^p(D_0)}
		\le  
		C_{bound}(p,D_0)\cdot \var(\eta,D_0)
		\cdot \\ \left(\Delta \norm{F(z)-F(z_0)}{L^p(\partial D_+)}
		+\frac{\Delta}{\delta}\norm{|F_1(z_0)| - |F_2(z_0)| }{L^p(\partial D_+)}\right)
		.
	\end{multline*}
	Applying the Trace theorem (Theorem \ref{thm:trace})
	we further get that
	\begin{multline*}
			\norm{F_2(z)-e^{\i\alpha}F_1(z)}{L^p(D_0)}
		\le  
		C_{bound}(p,D_0)\cdot \var(\eta,D_0)\cdot C_{trace}(p,D_+)
		\cdot \\ \left(\Delta \norm{F(z)-F(z_0)}{W^{1,p}( D_+)}
		+\frac{\Delta}{\delta}\norm{|F_1(z_0)| - |F_2(z_0)| }{L^p( D_+)}\right),
	\end{multline*}
	where we have used that $\||F_1(z_0)|-|F_2(z_0)| \|_{W^{1,p}(D_+)} = \||F_1(z_0)|-|F_2(z_0)| \|_{L^p(D_+)}$ because the function is constant.
	With the assumption in (\ref{eq:z0est-proof})
	we further get
	\begin{multline*}
			\norm{F_2(z)-e^{\i\alpha}F_1(z)}{L^p(D_0)}
		\le  
			C_{bound}(p,D_0)\cdot \var(\eta,D_0)\cdot C_{trace}(p,D_+)
		\cdot \\ \left(\Delta \norm{F(z)-F(z_0)}{W^{1,p}( D_+)}
		+\frac{\Delta}{\delta}C_{samp}\norm{|F_1| - |F_2| }{L^p( D_+)}\right)
		.
			\end{multline*}
			Lastly
			we apply the analytic Poincar\'e inequality (\ref{eq:Poincare}) and obtain
			the estimate
	\begin{multline}\label{eq:innerouter}
			\norm{F_2(z)-e^{\i\alpha}F_1(z)}{L^p(D_0)}
		\le  
			C_{bound}(p,D_0)\cdot \var(\eta,D_0)\cdot C_{trace}(p,D_+)
		\cdot \\ \left(\Delta(C^a_{poinc}(p,D_+,z_0)) \norm{F'}{L^p( D_+)}
		+\frac{\Delta}{\delta}C_{samp}\norm{|F_1| - |F_2| }{L^p( D_+)}\right)
		.
			\end{multline}

	{\bf Step 3. }
	In this step we focus on an estimate for the first term in 
	(\ref{eq:firstesta}).	
	Using (\ref{eq:basicest})  
	we see that
	$$
		\norm{F_2(z)-e^{\i\alpha}F_1(z)}{L^p(D_+)}\le  
		\Delta \norm{F(z) - F(z_0)}{L^p(D_+)}
		+\frac{\Delta}{\delta} 
		C_{samp}\norm{|F_1| - |F_2| }{L^p( D_+)}.
	$$	
	Yet another application of the analytic Poincar\'e inequality yields
	\begin{multline}\label{eq:outer}
		\norm{F_2(z)-e^{\i\alpha}F_1(z)}{L^p(D_+)}\le 
		\\ 
		\Delta C^a_{poinc}(p,D_+,z_0)\norm{F'}{L^p(D_+)}
		+\frac{\Delta}{\delta} 
		C_{samp}\norm{|F_1| - |F_2| }{L^p( D_+)}.
	\end{multline}
	{\bf Step 4. }In equations (\ref{eq:innerouter}) and (\ref{eq:outer})
	we now have achieved estimates of both terms in (\ref{eq:firstesta}).
	A close look at these estimates reveals that we only need to get 
	a bound on $\norm{F' }{L^p(D_+)}$ in terms of 
	$\norm{|F_1|-|F_2|}{W^{1,p}(D)}$ to finish the proof. This is where
	our key lemma, Lemma \ref{keylemma} comes into play, stating that
	$$
		\norm{F' }{L^p(D_+)}= \norm{\nabla |F| }{L^p(D_+)}.
	$$
	It thus remains to achieve a bound for $\norm{\nabla |F|}{L^p(D_+)}$.
	To this end we consider
	\begin{eqnarray*}
		\frac{\partial}{\partial x}|F| &=&
		\frac{|F_1| \frac{\partial}{\partial x}|F_2|  - |F_2|\frac{\partial}{\partial x}|F_1|}{|F_1|^2} \\
		&=&
		\frac{\frac{\partial}{\partial x}|F_1| (|F_1| - |F_2|)
		+ |F_1|(\frac{\partial}{\partial x}|F_2|-\frac{\partial}{\partial x}|F_1|)}{|F_1|^2},
	\end{eqnarray*}
	and hence,
	$$
		\left|\frac{\partial}{\partial x}|F|\right| \le \frac{\Delta}{\delta^2}\left(||F_1| - |F_2||+
		|\frac{\partial}{\partial x}|F_2|-\frac{\partial}{\partial x}|F_1||\right),
	$$
	valid uniformly on $D_+$.
	A similar estimate holds for $	\left|\frac{\partial}{\partial y}|F|\right|$ and thus there exists a universal constant $c>0$
	with 
	\begin{equation}\label{eq:laststep}
		\norm{F'}{L^p(D_+)}\le c\cdot \frac{\Delta}{\delta^2}\norm{|F_1|-|F_2|}{W^{1,p}(D_+)}.
	\end{equation}
	{\bf Step 5. }We finish the proof by substituting the estimate
	(\ref{eq:laststep}) into equations (\ref{eq:innerouter}) and (\ref{eq:outer}) (and noting that $\frac{\Delta}{\delta}\geq 1$), then use Lemma \ref{lem:tlem} to remove the dependency on $z_0$, which gives the desired result.
\end{proof}

\bibliographystyle{abbrv}
\bibliography{Deep_Learning.bib}

\begin{thebibliography}{10}

\bibitem{akutowicz1957determination}
E.~J. Akutowicz.
\newblock {On the determination of the phase of a Fourier integral, II}.
\newblock {\em Proceedings of the American Mathematical Society},
  8(2):234--238, 1957.

\bibitem{alaifari2016reconstructing}
R.~Alaifari, I.~Daubechies, P.~Grohs, and G.~Thakur.
\newblock Reconstructing real-valued functions from unsigned coefficients with
  respect to wavelet and other frames.
\newblock {\em Journal of Fourier Analysis and Applications}, pages 1--15,
  2016.

\bibitem{aifariunstable}
R.~Alaifari and P.~Grohs.
\newblock Gabor phase retrieval is severely ill-posed.
\newblock In preparation.

\bibitem{grohsstab}
R.~Alaifari and P.~Grohs.
\newblock {Phase retrieval in the general setting of continuous frames for
  Banach spaces}.
\newblock {\em To appear in SIAM Math Analysis, arXiv preprint
  arXiv:1604.03163}, 2016.

\bibitem{ascensi2009model}
G.~Ascensi and J.~Bruna.
\newblock {Model Space Results for the Gabor and Wavelet transforms}.
\newblock {\em IEEE Transactions on Information Theory}, 5(55):2250--2259,
  2009.

\bibitem{balan2006signal}
R.~Balan, P.~Casazza, and D.~Edidin.
\newblock On signal reconstruction without phase.
\newblock {\em Applied and Computational Harmonic Analysis}, 20(3):345--356,
  2006.

\bibitem{bandeira2014saving}
A.~S. Bandeira, J.~Cahill, D.~G. Mixon, and A.~A. Nelson.
\newblock Saving phase: Injectivity and stability for phase retrieval.
\newblock {\em Applied and Computational Harmonic Analysis}, 37(1):106--125,
  2014.

\bibitem{cahill2016phase}
J.~Cahill, P.~Casazza, and I.~Daubechies.
\newblock {Phase retrieval in infinite-dimensional Hilbert spaces}.
\newblock {\em Transactions of the American Mathematical Society, Series B},
  3(3):63--76, 2016.

\bibitem{candes2015phase}
E.~J. Cand{\`e}s, Y.~C. Eldar, T.~Strohmer, and V.~Voroninski.
\newblock Phase retrieval via matrix completion.
\newblock {\em SIAM Review}, 57(2):225--251, 2015.

\bibitem{candes2015phasewirt}
E.~J. Candes, X.~Li, and M.~Soltanolkotabi.
\newblock Phase retrieval via wirtinger flow: Theory and algorithms.
\newblock {\em IEEE Transactions on Information Theory}, 61(4):1985--2007,
  2015.

\bibitem{chen2016phase}
Y.~Chen, C.~Cheng, Q.~Sun, and H.~Wang.
\newblock {Phase Retrieval of Real-Valued Signals in a Shift-Invariant Space}.
\newblock {\em arXiv preprint arXiv:1603.01592}, 2016.

\bibitem{deller1993discrete}
J.~R. Deller~Jr, J.~G. Proakis, and J.~H. Hansen.
\newblock {\em Discrete time processing of speech signals}.
\newblock Prentice Hall PTR, 1993.

\bibitem{evans}
L.~C. Evans.
\newblock {\em Partial Differential Equations}.
\newblock Graduate studies in mathematics. American Mathematical Society,
  Providence (R.I.), 1998.

\bibitem{fienup1982phase}
J.~R. Fienup.
\newblock Phase retrieval algorithms: a comparison.
\newblock {\em Applied optics}, 21(15):2758--2769, 1982.

\bibitem{flanagan1966phase}
J.~L. Flanagan and R.~Golden.
\newblock Phase vocoder.
\newblock {\em Bell System Technical Journal}, 45(9):1493--1509, 1966.

\bibitem{gerchberg1972practical}
R.~W. Gerchberg.
\newblock A practical algorithm for the determination of phase from image and
  diffraction plane pictures.
\newblock {\em Optik}, 35:237, 1972.

\bibitem{gottlieb1985eigenvalues}
H.~Gottlieb.
\newblock {Eigenvalues of the Laplacian with Neumann boundary conditions}.
\newblock {\em The Journal of the Australian Mathematical Society. Series B.
  Applied Mathematics}, 26(03):293--309, 1985.

\bibitem{grebenkov2013geometrical}
D.~S. Grebenkov and B.-T. Nguyen.
\newblock {Geometrical structure of Laplacian eigenfunctions}.
\newblock {\em SIAM Review}, 55(4):601--667, 2013.

\bibitem{grochenig}
K.~Gr{\"o}chenig.
\newblock {\em Foundations of time-frequency analysis}.
\newblock Springer Science \& Business Media, 2013.

\bibitem{grohsrathmair}
P.~Grohs and M.~Rathmair.
\newblock Stable {Gabor} phase retrieval.
\newblock In preparation.

\bibitem{hempel2006lowest}
R.~Hempel.
\newblock {On the lowest eigenvalue of the Laplacian with Neumann boundary
  condition at a small obstacle}.
\newblock {\em Journal of computational and applied mathematics},
  194(1):54--74, 2006.

\bibitem{humphry2012ptychographic}
M.~Humphry, B.~Kraus, A.~Hurst, A.~Maiden, and J.~Rodenburg.
\newblock Ptychographic electron microscopy using high-angle dark-field
  scattering for sub-nanometre resolution imaging.
\newblock {\em Nature communications}, 3:730, 2012.

\bibitem{hurt2001phase}
N.~E. Hurt.
\newblock {\em Phase retrieval and zero crossings: mathematical methods in
  image reconstruction}, volume~52.
\newblock Springer Science \& Business Media, 2001.

\bibitem{katznelson}
Y.~Katznelson.
\newblock {\em An Introduction to Harmonic Analysis}.
\newblock Cambridge University Press, 2004.

\bibitem{klibanov1986inverse}
M.~V. Klibanov.
\newblock Inverse scattering problems and restoration of a function from the
  modulus of its fourier transform.
\newblock {\em Siberian Mathematical Journal}, 27(5):708--719, 1986.

\bibitem{laroche1999improved}
J.~Laroche and M.~Dolson.
\newblock Improved phase vocoder time-scale modification of audio.
\newblock {\em IEEE Transactions on Speech and Audio processing},
  7(3):323--332, 1999.

\bibitem{mallat2012group}
S.~Mallat.
\newblock Group invariant scattering.
\newblock {\em Communications on Pure and Applied Mathematics},
  65(10):1331--1398, 2012.

\bibitem{marchesini2003x}
S.~Marchesini, H.~He, H.~N. Chapman, S.~P. Hau-Riege, A.~Noy, M.~R. Howells,
  U.~Weierstall, and J.~C. Spence.
\newblock X-ray image reconstruction from a diffraction pattern alone.
\newblock {\em Physical Review B}, 68(14):140101, 2003.

\bibitem{payne1960optimal}
L.~E. Payne and H.~F. Weinberger.
\newblock {An optimal Poincar{\'e} inequality for convex domains}.
\newblock {\em Archive for Rational Mechanics and Analysis}, 5(1):286--292,
  1960.

\bibitem{pohl2014phaseless}
V.~Pohl, F.~Yang, and H.~Boche.
\newblock Phaseless signal recovery in infinite dimensional spaces using
  structured modulations.
\newblock {\em Journal of Fourier Analysis and Applications}, 20(6):1212--1233,
  2014.

\bibitem{rodenburg2008ptychography}
J.~Rodenburg.
\newblock Ptychography and related diffractive imaging methods.
\newblock {\em Advances in Imaging and Electron Physics}, 150:87--184, 2008.

\bibitem{shechtman2015phase}
Y.~Shechtman, Y.~C. Eldar, O.~Cohen, H.~N. Chapman, J.~Miao, and M.~Segev.
\newblock Phase retrieval with application to optical imaging: a contemporary
  overview.
\newblock {\em Signal Processing Magazine, IEEE}, 32(3):87--109, 2015.

\bibitem{poincare}
A.~Stanoyevitch and D.~A. Stegenga.
\newblock {Equivalence of analytic and Sobolev Poincar{\'e} inequalities for
  planar domains}.
\newblock {\em Pacific Journal of Mathematics}, 178(2):363--375, 1997.

\bibitem{thakur2011reconstruction}
G.~Thakur.
\newblock Reconstruction of bandlimited functions from unsigned samples.
\newblock {\em Journal of Fourier Analysis and Applications}, 17(4):720--732,
  2011.

\bibitem{waldspurger2015these}
I.~Waldspurger.
\newblock {\em Wavelet transform modulus: phase retrieval and scattering}.
\newblock PhD thesis, {\'E}cole normale sup{\'e}rieure, 2015.

\bibitem{waldspurger2015phase}
I.~Waldspurger, A.~d’Aspremont, and S.~Mallat.
\newblock {Phase recovery, MaxCut and complex semidefinite programming}.
\newblock {\em Mathematical Programming}, 149(1-2):47--81, 2015.

\bibitem{wang1983real}
H.-C. Wang.
\newblock Real {H}ardy spaces of an annulus.
\newblock {\em Bull. Australian Math. Soc}, 27:91--105, 1983.

\bibitem{zheng2013wide}
G.~Zheng, R.~Horstmeyer, and C.~Yang.
\newblock Wide-field, high-resolution fourier ptychographic microscopy.
\newblock {\em Nature photonics}, 7(9):739--745, 2013.

\end{thebibliography}

\end{document}